\documentclass{siamltex}
\usepackage{amsmath,amsfonts,amssymb}
\usepackage{latexsym}
\usepackage{graphicx,overpic}
\usepackage{subfigure,caption}
\usepackage{pgfplots}
\pgfplotsset{compat=newest}
\usepackage{color}
\usepackage{booktabs}
\usepackage{color}
\usepackage{tikz}

\newcommand{\tr}{{\rm tr\,}}
\newcommand{\R}{\mathbb R}
\newcommand{\NN}{\mathbb N}
\newcommand{\TT}{\mathbb T}
\newcommand{\PP}{\mathcal P}

\newcommand{\bI}{\mathbf I}

\newcommand{\bL}{\mathbf L}
\newcommand{\cV}{\mathcal{V}}
\newcommand{\cI}{\mathcal{I}}

\newcommand{\cQ}{\mathcal{Q}}

\newcommand{\cU}{\mathcal{U}}

\newcommand{\Q}{\mathcal Q}

\newcommand{\V}{\mathcal V}
\newcommand{\Div}{\mathop{\rm div}}

\newcommand{\cH}{\mathcal H}

\newcommand{\cD}{\mathcal D}

\newcommand{\cO}{\mathcal O}

\newcommand{\cP}{\mathcal P}

\newcommand{\nn}{\mathbb{N}}

\renewcommand{\div}{\textrm{div}\ \!}

\DeclareGraphicsExtensions{.pdf,.eps,.ps,.eps.gz,.ps.gz,.eps.Y}

\newtheorem{mydef}{Definition}[section]

\newtheorem{assumption}{Assumption}[section]

\newtheorem{example}{Example}[section]

\newtheorem{remark}{Remark}[section]

\begin{document}
\definecolor{RED}{RGB}{255,0,0}
\title{Discontinuous Galerkin Time Discretization Methods for Parabolic Problems with {Linear} Constraints}
\author{Igor Voulis\thanks{Institut f\"ur Geometrie und Praktische  Mathematik, RWTH-Aachen
University, D-52056 Aachen, Germany (voulis@igpm.rwth-aachen.de)} \and Arnold Reusken\thanks{Institut f\"ur Geometrie und Praktische  Mathematik, RWTH-Aachen
University, D-52056 Aachen, Germany (reusken@igpm.rwth-aachen.de)}
}
\maketitle
\begin{abstract} 
We consider time discretization methods for abstract parabolic problems with inhomogeneous linear constraints. Prototype examples that fit into the general framework are the heat equation with inhomogeneous (time-dependent) Dirichlet boundary conditions and the time-dependent Stokes equation with an inhomogeneous divergence constraint.  Two common ways of treating such linear constraints, namely explicit or implicit (via Lagrange multipliers) are studied.  These different treatments lead to different variational formulations of the parabolic problem. For these formulations we introduce a modification of the standard discontinuous Galerkin (DG) time discretization method in which an appropriate projection is used in the discretization of the constraint. 
For these discretizations (optimal)  error bounds, including superconvergence results,  are derived.  Discretization error bounds for the Lagrange multiplier are presented. Results of experiments confirm the theoretically predicted optimal convergence rates and show that without the modification the (standard) DG method has sub-optimal convergence behavior. 
\end{abstract}
\section{Introduction}
 Nowadays the discontinuous Galerkin (DG) finite element method is a popular discretization technique for many classes of ordinary and partial differential equations, cf. the overviews in  \cite{ErnGuermond,hesthaven,DiPietro}. For ordinary differential equations  DG finite element methods have been analyzed in e.g. \cite{Delfour1981,Estep1995,Johnson1988} and the references therein. For parabolic partial differential equations the DG finite element for time discretization has been introduced in \cite{Jamet1978} and was further studied in \cite{Eriksson1991,Eriksson1995,Eriksson1998,Larsson1998,Thomee, Schwab,Matthies}. In this paper we  study DG finite element \emph{time} discretization methods for \emph{parabolic problems with  constraints}.
As far as we know, in the DG error analyses for parabolic problems available in the literature only the case with \emph{homogeneous} constraints has been considered.  In practice one often has to deal with \emph{in}homogeneous constraints. Typical examples are a heat equation with nonzero time-dependent Dirichlet boundary conditions, an instationary Stokes flow with a time-dependent inflow boundary condition and a Stokes flow with an inhomogeneous time-dependent divergence constraint. Runge-Kutta methods for parabolic equations with a non-homogenous constraint have been analyzed  in the recent paper \cite{AltmannZimmer}. 
{A DG time discretization method for the wave equation with an inhomogenous boundary condition is analyzed in \cite{Walkington2014}. }
In this paper  we consider a DG time discretization of a parabolic problem with non-homogeneous linear constraints and present an error analysis resulting in optimal error bounds. It turns out that for obtaining such optimal bounds we have to 
\emph{modify the standard formulation of a DG Galerkin finite element method} as given in \cite{Thomee}. Results of numerical experiments will show that without an appropriate modification the standard 
DG Galerkin finite element method applied to a parabolic problem with an inhomogeneous constraint does \emph{not} yield optimal convergence rates.  

In this paper the parabolic problem and the error analysis of the DG time discretization are presented in an abstract setting (as in e.g.~\cite{Schwab}). The linear constraints that are part of the problem can be treated in different ways: explicitly, by eliminating variables, implicitly, by  means of a Lagrange-multiplier, or by a combination of these. The analysis in the paper covers all three cases. Prototype examples that fit into the abstract framework are the standard scalar heat equation with inhomogeneous (time-dependent) Dirichlet boundary conditions (usually treated explicitly) and the instationary Stokes equations with  inhomogeneous (time-dependent) Dirichlet boundary conditions (usually treated explicitly) and/or an inhomogeneous divergence constraint (usually treated implicitly by the pressure Lagrange multiplier). 
These examples will be discussed in detail in Subsection~\ref{sectexamples}.


The main results of the paper are the following. Firstly, we present a \emph{modification of the standard formulation} of a DG time discretization method that yields optimal order discretization errors also for the case of \emph{in}homogeneous constraints.  Secondly, for this modified method we \emph{prove optimal discretization error bounds}.  Both optimal  global energy norm bounds and optimal superconvergence results are derived. 
Finally, for the case that the constraints are treated implicitly (e.g. divergence constraint in Stokes problem) we derive \emph{discretization error bounds for the corresponding Lagrange multiplier}. In the example of the Stokes problem this yields error bounds for the DG time discretization of the pressure variable. We note that even for the case of homogeneous constraints we are not aware of error analyses of DG time discretization methods for Stokes equations that yield (optimal order) error bounds for the pressure variable. A more detailed discussion of how the results of this paper are related to other literature is given in Remark~\ref{remliterature}.

The paper is organized as follows. In Section~\ref{SectionProblem} we introduce the class of abstract parabolic problems that we consider. Different formulations, related to the different treatments of the constraints, are presented. Furthermore, specific concrete examples that fit in the abstract setting are discussed. In Section~\ref{SectionDiscDG} we introduce a  DG scheme for the time discretization of the abstract parabolic problem.  In Section~\ref{SectionErrorAnalysis} we give an error analysis of the DG scheme.  We derive optimal order error bounds in the energy norm and prove a nodal superconvergence result. In Section~\ref{SectionErrorMixed}, for the case of an implicit treatment of the constraint, we analyze the discretization error in the Lagrange multiplier. We derive, under very mild assumptions,  a sub-optimal result for the convergence of the Lagrange multiplier and using an additional  regularity condition we are able to prove an optimal discretization error bound. In Section~ \ref{SectionFullDisc} we introduce a fully discrete numerical scheme. In Section~\ref{SectionNum} this scheme is used to perform numerical experiments and results of a few experiments are presented that illustrate the convergence behavior of the DG method.  Finally, we give an outlook in Section \ref{SectionOutlook}. 

\section{Parabolic problem with linear constraints}\label{SectionProblem}
In this section we introduce the class of parabolic problems with linear constraints that we treat in this paper. For this we first introduce some notation and recall relevant well-known results on well-posedness of (abstract) parabolic problems e.g. \cite{Thomee,Schwab}.  Let $\cU,\cH$ be real separable Hilbert spaces with a dense continuous embedding $\cU \hookrightarrow \cH$. The norms are denoted by $\|\cdot\|_\cU,\, \|\cdot\|_\cH$, respectively. These spaces induce a Gelfand triple $\cU \hookrightarrow \cH \cong \cH' \hookrightarrow \cU'$. Let $a: \cU \times \cU \to \Bbb{R}$ be a symmetric continuous coercive bilinear form on $\cU$:
\begin{align}
  |a(u,v)| & \leq \Gamma \|u\|_\cU \|v\|_\cU \quad \text{for all}~~u,v \in \cU, \label{contOld}
  \\
  a(v,v) & \geq \gamma \|v\|_\cU^2 \quad \text{for all}~~v \in \cU, \label{ellipOld}
\end{align}
with $\gamma >0$. The corresponding operator is denoted by $A:\, \cU \to \cU'$, $Au(v)=a(u,v)$. Let $I=(0,T)$ be a given time interval and for Hilbert spaces $Z_1,Z_2$ we define ${W^1(Z_1;Z_2)}:=\{\, u \in L^2(I;Z_1)~|~u' \in L^2(I;Z_2)\,\}$, $H^m(I;Z_1):=H^m(I)\otimes Z_1$,  cf.~\cite{Wloka}. A standard weak formulation of an abstract parabolic problem is as follows: find $u \in {W^1(\cU;\cU')}$ such that $u(0)=u_0$ and
\begin{equation} \label{parabolic1}
 u' + A u=f \quad \text{in}~~L^2(I;\cU').
\end{equation}
We assume $u_0 \in \cH$, $f \in L^2(I;\cU')$. This problem is well-posed \cite{Wloka,Schwab}. Standard examples of problems that fit in this abstract framework are the heat equation with \emph{homogeneous} BC ($\cU=H_0^1(\Omega)$) and the Stokes equation with \emph{homogeneous} BC ($\cU=\{\, u \in H_0^1(\Omega)^d~|~\div u= 0\,\}$), cf.~Examples~\ref{Exheat} and \ref{ExStokes} below.  In the literature one can find analyses of DG time discretizations applied to these parabolic problems \cite{Thomee,Schwab}. In this paper we present an extension of the time discretization error analysis that applies to such problems with \emph{in}homogeneous boundary conditions or an \emph{in}homogeneous divergence constraint $\div u= g$. For this we  introduce an abstract (variatiational) formulation of a problem as in \eqref{parabolic1} with  linear constraints. For a  Hilbert space $\Q$ let  $b(\cdot,\cdot): \, \cU \times \cQ \to \R$ be a continuous bilinear form and $B:\cU \to \cQ'$ the corresponding operator 
representation: $Bu(q)=b(u,q)$ for all $u \in \cU$, $q \in \cQ$. This bilinear form and linear operator $B$ are called  \emph{constraint operators}.  We define 
\[ \cV:= \ker (B)= \{\, u \in \cU~|~ Bu=0\,\},
\]
 which is a Hilbert space. 
 We assume that
\begin{equation} \label{assumpB}
  B:\, \cU \to \cQ' \quad \text{is surjective}.
\end{equation}
This is equivalent (\cite{ErnGuermond}, Lemma A.40) to the inf-sup property:
\begin{equation} \label{assumpB1}
\exists ~~\alpha>0: ~~\inf_{q \in\cQ} \sup_{u \in \cU} \frac{b(u,q)}{\|u\|_{\cU} \|q\|_\cQ} \geq \alpha.
\end{equation}
In the subsections below we give different formulations for the abstract parabolic problem with constraints which are relevant in particular applications, cf. the examples in Subsection~\ref{sectexamples}.
\subsection{Constrained formulation}  
For given $f \in L^2(I;\cV')$, $g \in L^2(I;\cQ')$ we consider the following abstract parabolic problem: find $u \in {W^1(\cU;\cV')}$ such that $u(0)=u_0$ and
\begin{equation} \label{problem}
 \begin{split}
   u'+A u & = f \quad \text{in}~~L^2(I;\cV') \\
    Bu &=g \quad \text{in}~~L^2(I;\cQ').
 \end{split}
\end{equation}
We relax the ellipticity condition {\eqref{ellipOld}} and only require ellipticity of $a(\cdot,\cdot)$ to hold on the subspace $\cV \subset \cU$. We thus have 
\begin{align}
  |a(u,v)| & \leq \Gamma \|u\|_\cU \|v\|_\cU \quad \text{for all}~~u,v \in \cU, \label{cont}
  \\
  a(v,v) & \geq \gamma \|v\|_\cU^2 \quad \text{for all}~~v \in \cV, \label{ellip}
\end{align}
with $\gamma,\Gamma >0$.
Concerning existence of a unique solution of \eqref{problem} we note the following. If $g=0$ then the constraint in \eqref{problem} implies $u \in \cV$, and thus we are in the standard setting described above, with $\cU$ replaced by $\cV$ (and the Gelfand triple $\cV \hookrightarrow \overline{\cV}^{\cH} \hookrightarrow \cV'$). Hence we have well-posedness. For the inhomogeneous case we apply a standard translation argument. Note that from the inf-sup property it follows that there exists a unique $v \in L^2(I;\cV^\perp)$ such that $Bv = g$, where $\cV^{\perp}$ is the orthogonal complement of $\cV$ in $\cU$. Assume that $g$ has sufficient regularity, e.g. $g\in H^1(I;\cQ')$, 
 such that there exists a
$v\in {W^1(\cU;\cV')}$ that satisfies $Bv=g$.  
Application of 
the standard analysis yields that there exists (a unique) $w \in {W^1(\cV;\cV')}$ with $w(0)=u_0-v(0)$ and
\[
w' + A w=f- A v-v' \quad \text{in}~~L^2(I;\cV').
\]
Hence $u:=w+v  \in W^1(\cU;\cV') $ satisfies $u(0)=u_0$   and the variational equation \eqref{problem}.  Uniqueness follows from the fact that \eqref{problem} with $g=0$, $f=0$ and $u_0=0$ has only the trivial solution.  
\subsection{Mixed formulation} \label{sectmixed} In practice, in particular in the context of Galerkin finite element discretization methods,  it may be not convenient to use the kernel space $\cV$ in the first equation in \eqref{problem}. One then introduces a suitable Lagrange multiplier to enlarge the constrained space $\cV$. We introduce such a mixed formulation. The structure of this mixed formulation is inspired by the examples considered in Subsection~\ref{sectexamples}, in which only for a part of the constraint equation $Bu=g$ a corresponding Lagrange multiplier is used. For this we decompose the constraint bilinear form as
\[
  b(u,q)=b_1(u,q_1) + b_2(u,q_2), \quad  q= (q_1,q_2) \in \cQ_1 \times \cQ_2 =\cQ,
\]
with $\cQ_i$, $i=1,2$, Hilbert spaces, and $b_i:\, \cU \times \cQ_i \to \Bbb{R}$ continuous bilinear forms. This splitting is such that for the $b_1(\cdot,\cdot)$ part a Lagrange multiplier will be introduced, whereas for the $b_2(\cdot,\cdot)$ part no Lagrange multiplier is used. We assume that $\cQ_1 \neq 0$. The bilinear forms $b_i(\cdot,\cdot)$ induce corresponding operators $B_i:\, \cU \to \cQ_i'$. From \eqref{assumpB1} it follows that 
\begin{equation} \label{assumpB1A}
 \inf_{q_i \in\cQ_i} \sup_{u \in \cU} \frac{b_i(u,q_i)}{\|u\|_{\cU} \|q_i\|_{\cQ_i}} \geq \alpha, \quad i=1,2,
\end{equation}
holds, provided $b_i(\cdot,\cdot)$ is not identically zero.  Define
\[
 \cV_i:=\ker (B_i)= \{\, u \in \cU~|~b_i(u,q_i)=0\quad  \forall ~q_i \in \cQ_i\,\}, \quad i=1,2.
\]
Hence, $\cV=\cV_1 \cap \cV_2$.  We also have
\begin{equation} \label{newid}
  \cV=\{\, u \in \cV_2~|~b_1(u,q_1)=0 \quad \forall~ q_1 \in  \cQ_1\,\}.
\end{equation}
We allow  $\cQ_2 = \{0\}$, $\cQ_1=\cQ$, in which  case we have $\cV_2=\cU,~\cV=\cV_1$. Note that $B:\cV^\perp \rightarrow \cQ_1' \times \cQ_2'$ is a bijection and the preimage of $\cQ_1'\times \{0\}$ is $\cV^\perp\cap \cV_2$. It follows that $B_1:\cV^\perp\cap \cV_2 \rightarrow \cQ_1'$ is a bijection, from which we conclude that $B_1: \cV_2 \rightarrow \cQ_1'$ is surjective. This is equivalent (\cite{ErnGuermond}, Lemma A.40) to the inf-sup property:
\begin{equation} \label{assumpB1AA}
 \exists ~~\beta>0: ~~\inf_{q_1 \in\cQ_1} \sup_{u \in \cV_2} \frac{b_1(u,q_1)}{\|u\|_{\cU} \|q_1\|_{\cQ_1}} \geq \beta.
\end{equation}
Note that if $\cQ_2 = \{0\}$ then \eqref{assumpB1AA} is the same as \eqref{assumpB1A} with $i=1$. From \eqref{assumpB1AA}  and \eqref{newid} it follows that for the adjoint $B_1'$ of $B_1:\, \cV_2 \to \cQ_1'$ we have that
\begin{equation} \label{iso1}
 B_1' :\, Q_1 \to \cV_2^0:= \{\, g \in \cV_2'~|~g(v)=0 \quad \text{for all}~~v \in \cV\, \}
\end{equation}
is an isomorphism and
\begin{equation} \label{iso2}
 \|B_1' p\|_{\cV_2'} \geq \beta \|p\|_{Q_1} \quad \text{for all}~~p \in Q_1,
\end{equation}
cf. Lemma 4.1 in \cite{GiraultRaviart}. We now introduce a mixed formulation of \eqref{problem}. For $f \in L^2(I;\cV_2')$ determine $u \in W^1(\cU;\cV_2')$, $p\in L^2(I;\cQ_1)$ such that
\begin{equation} \label{problemmixed}
 \begin{split}
   u'+A u + B_1' p& = f \quad \text{in}~~L^2(I;\cV_2') \\
    Bu &=g \quad \text{in}~~L^2(I;\cQ').
 \end{split}
\end{equation}
Here $B_1': \cQ_1 \to \cV_2'$ is the adjoint of $B_1:\, \cV_2 \to \cQ_{1}'$. Note that for $v \in \cV_2$ we have $b_2(v,q_2)=0$ for all $q_2 \in \cQ_2$, hence, $(B_1'p)(v)=b_1(v,p)=b(v,(p,0))$ holds for $v \in \cV_2$, $p \in \cQ_1$.

 We now discuss the relation between the constrained formulation \eqref{problem} and the mixed formulation \eqref{problemmixed}. Let $u\in W^1(\cU; \cV')$ be the unique solution of \eqref{problem}. We \emph{assume} that $u$ has additional regularity such that $u\in W^1(\cU; \cV_2')$. Then $\ell(t):= f(t)-u'(t)- Au(t) \in \cV_2'$ (a.e. in $t$) and $\ell(t)(v)=0$ for all $v \in \cV$, i.e., $\ell(t)\in \cV_2^0$ holds. Hence, there exists a unique $p(t) \in \cQ_1$ such that $B_1'p(t) = \ell(t)$. Using \eqref{iso2} it follows that $p \in  L^2(I;\cQ_1)$. Thus, if the unique solution $u$ of the constrained formulation \eqref{problem} has regularity $u \in W^1(\cU; \cV_2')$, there is a unique $p \in L^2(I; \cQ_1)$ such that $(u,p)$ is the unique solution of the mixed formulation   \eqref{problemmixed}.
\subsection{Examples} \label{sectexamples}
 We describe three examples that fit into the general abstract framework presented above. In these examples we consider different types of constraints, which are allowed to be inhomogeneous.  \\
\begin{example}[Heat equation] \label{Exheat} {\rm The heat equation in $\Omega\subset \Bbb{R}^d$, with \emph{(time-dependent) inhomogeneous Dirichlet boundary conditions} $h\in L^2(I;H^\frac12(\partial \Omega))$ in constrained formulation is as follows:
\begin{eqnarray*}
u' - \Delta u &=& f \quad \text{in}~~L^2(I;H^{-1}(\Omega))\\
u|_{\partial \Omega} &=& h.
\end{eqnarray*}
For this problem we have $\cU=H^1(\Omega)$, $\cH=L^2(\Omega)$, $\cQ=H^\frac12(\partial \Omega)$, $a(u,v):=\int_{\Omega} \nabla u \cdot \nabla v$, $b(u,q)=(\tr(u),q)_{H^\frac12(\partial \Omega)}$, with $\tr:\, H^1(\Omega)  \to H^\frac12(\partial \Omega)$ the trace operator, which is surjective. This yields
$\cV=\ker(B)=H_0^1(\Omega)$. The assumptions \eqref{assumpB}, \eqref{cont}, \eqref{ellip} are satisfied. We do not consider a mixed formulation, because in  the constrained formulation \eqref{problem} we use the space $\cV= H_0^1(\Omega)$, which is the natural one for the bilinear form $a(\cdot,\cdot)$. 
}\end{example}
\ \\[1ex]
\begin{example}[Stokes equations]\label{ExStokes} {\rm
 We consider the Stokes equation in $\Omega\subset \Bbb{R}^d$ with  homogeneous boundary conditions and a \emph{non-zero (time-dependent) divergence}  $g \in L^2(I;L_0^2(\Omega))$. In this setting we have  $\cU =H_0^1(\Omega)^d$, $\cH=L^2(\Omega)^d$, $\cQ=L_0^2(\Omega):=L^2(\Omega)/\R$, $a(u,v):=\int_{\Omega} \nabla u : \nabla v$, $b(u,q)=\int_\Omega \Div u\, q$, $\cV=\{\, u \in H_0^1(\Omega)^d~|~\Div u =0\,\}$. Recall that $u \in H_0^1(\Omega)^d \to \Div u \in L^2_0(\Omega)$ is surjective. The Stokes equation in  constrained formulation is as follows:
\begin{eqnarray*}
u' -\Delta u &=& f \quad \text{in}~~ L^2(I;\cV') \\
\Div u &=& g \quad \text{in}~~ L^2(I;L_0^2(\Omega)).
\end{eqnarray*}
The assumptions \eqref{assumpB}, \eqref{cont}, \eqref{ellip} are satisfied. For the mixed formulation we take $b_2 \equiv 0$, i.e., $b_1(u,q)=\int_\Omega \Div u\, q$, $q \in \cQ_1=\cQ$, and $\cV_2=\cU=H_0^1(\Omega)^d$. Hence, the formulation \eqref{problemmixed} is the standard mixed formulation of a time-dependent Stokes problem with homogeneous boundary conditions, cf. \cite{ErnGuermond} Sect. 4.1. 
} \end{example}
\ \\[1ex]
\begin{example} \label{example3} {\rm 
We consider the Stokes equation in $\Omega\subset \Bbb{R}^d$  with  \emph{non-zero (time-dependent) boundary data $h \in L^2(I;H^\frac12(\partial \Omega)^d)$ and a non-zero (time-dependent) divergence} $g  \in L^2(I;L_0^2(\Omega))$.  We then have $\cU=H^1(\Omega)^d$, $\cH=L^2(\Omega)^d$, $\cQ=L_0^2(\Omega) \times H^\frac12(\partial \Omega)^d$, $a(u,v):=\int_{\Omega} \nabla u : \nabla v$, $b(u,(q_1,q_2))=\int_\Omega \Div u\, q_1 + (\tr u, q_2)_{H^\frac12(\partial \Omega)}$, $\cV=\{\, u \in H_0^1(\Omega)^d~|~\Div u =0\,\}$. Note that $u \in H^1(\Omega)^d \to (\Div u, \tr u) \in L^2_0(\Omega) \times H^\frac12(\Omega)^d$ is surjective. The Stokes equation in constrained formulation is given by 
\begin{eqnarray*}
u' -\Delta u &=& f \quad \text{in}~~ L^2(I;\cV') \\
\Div u &=& g \quad \text{in}~~  L^2(I;L_0^2(\Omega))\\
\tr u &=& h \quad \text{in}~~  L^2(I;H^\frac12(\partial \Omega)^d).
\end{eqnarray*}
The assumptions \eqref{assumpB}, \eqref{cont}, \eqref{ellip} are satisfied. For the mixed formulation we take $\cQ_1=L_0^2(\Omega)$, $\cQ_2=H^\frac12(\Omega)^d$, $b_1(u,q_1)=\int_\Omega \Div u\, q_1$, $b_2(u,q_2)=(\tr u, q_2)_{H^\frac12(\partial \Omega)}$. Note that $\cV_2=H_0^1(\Omega)^d$. 
}\end{example}
\ \\[1ex]
\begin{remark} \label{remliterature} \rm
In the numerical approximation of these parabolic problems one has to consider discretization in space and in time. In this paper we focus on the time discretization.  We  now briefly discuss the space discretization and continue this discussion in Section~\ref{SectionFullDisc}.
The two formulations \eqref{problem} and \eqref{problemmixed} lead to two different numerical schemes. In the former case one has to replace the space $\cV$ by a suitable discrete (e.g., finite element) space. In the mixed formulation the spaces $\cV_2$ and $\cQ_1$ have to be replaced by a suitable pair of discrete (finite element) spaces.

In the specific case of the Stokes problem this leads to two different well-known finite element approaches. Numerous conforming and non-conforming finite element (or wavelet) schemes for space discretization of the Stokes problem  exist. By far most of these schemes use the mixed formulation \eqref{problemmixed}, cf. \cite{GiraultRaviart,ErnGuermond,Quarteroni,Braess,Elman2014}. In that setting  we need a pair of discrete spaces $\cV^h_2,\cQ_1^h$ for approximation of functions from the spaces $\cV_2=H^1_0(\Omega)^d$ and $\cQ_1=L^2_0(\Omega)$. Extensive literature concerning the choice of such a pair of spaces is  available, cf. \cite{ErnGuermond,GiraultRaviart} and the references therein.  Recently so-called pressure robust mixed methods  have been studied, e.g. \cite{LLS_ARXIV_2017,Brennecke,Linke1}. 
To treat the Stokes problem by means of the constrained formulation \eqref{problemmixed} a discrete (finite element or wavelet) space $\cV^h$ with divergence-free functions is required. Examples of such discrete spaces are treated in, e.g., \cite{Stevenson2016,Zhang,Falk,Guzman,Lehrenfeld2016}.
By far most of the error analyses of these methods (for both approaches) known in the liteature result in bounds for the \emph{spatial} discretization error in case of a stationary problem  and for the discretization error in the time-dependent \emph{semidiscrete} (i.e., discrete in space only) problem.  In only very few papers, e.g., \cite{Matthies,Linke2018}, the error in the fully discrete problem is treated. In these papers, however, only the case with homogeneous constraints (i.e., $\Div u=0$ and homogeneous Dirichlet boundary conditions) is treated and the error in the discretization of the Lagrange multiplier (pressure) is not considered or only suboptimal error bounds are derived.  

We also note that instead of these popular Rothe or method of lines techniques one can apply a direct full space-time approach, e.g. the recent papers \cite{SchwabStevenson,Steinbach2018}.

In this paper we do not compare all these different approaches. We present and analyze a  finite element method that is of Rothe type. For discretization in time we apply a DG finite element method. An important  aspect is that this method, which is a modification of the standard DG time discretization method (\cite{Thomee}) allows the treatment of time-dependent inhomogeneous constraints with an optimal order of accuracy. For the semidiscrete problem (i.e., discrete in time only) optimal discretization error bounds for both approaches (constrained and mixed) are derived. In the mixed formulation we derive optimal bounds for both the primal variable $u$ and the Lagrange multiplier $p$.
\end{remark}

\begin{remark}  \label{remgenral} \rm 
 \rm In the analysis it is essential that the constraint operator $B$ does not depend on $t$. Above we also assumed that $A$  is independent of  $t$. This assumption, however, is only essential for superconvergence and Subsection \ref{SubSectionBadLM}, all other results remain valid for a time-dependent $A$. Furthermore, the analysis also applies to a generalization of the problem \eqref{problem} where the term $u'(t)$ is replaced by $M u'(t)$ with $M$ a symmetric, elliptic bounded linear operator on $\cH$. Without loss of generality we can assume that $Mu(v) = (u,v)_\cH$. 
 {An example of the latter generalization is given  in Subsection \ref{SubSectionStokesInterface}.} 
\end{remark}


\section{Discontinuous Galerkin discretization} \label{SectionDiscDG}
In this section we present DG time discretization methods for the variational problems \eqref{problem} and \eqref{problemmixed}. First some notation is introduced.  We take a fixed $q \in \Bbb{N}$, $q \geq 1$. In the discretization we will use  polynomials of degree $q-1$ ($q$ degrees of freedom) in time. The space of  polynomials of degree $q-1$  is denoted by $\PP_{q-1}$. 
For $N\in \nn$, introduce $0=t_0<\dots<t_N=T$, $I_n=(t_{n-1},t_n]$,  and $k_n=|I_n|$ for $n=1,\dots,N$. For simplicity, we assume that $k=\max_n k_n \leq 1$.
We define the broken spaces
\begin{align*}
\PP^b(I) & = \bigoplus_{n=1}^N \PP_{q-1}(I_n) \subset L^2(I), \quad H^{1,b}(I) := \bigoplus_{n=1}^N H^1(I_n) \\
\PP^b(I;\hat\cH) &:=\PP^b(I)\otimes \hat\cH, \quad H^{1,b}(I;\hat\cH):=H^{1,b}\otimes \hat\cH,
\end{align*}
with $\hat\cH$ a  given Hilbert space.
For $U\in H^{1,b}(I;\hat\cH)$ and $n=1,\dots, N$ we will write $U^n=U|_{I_n}(t_n)$, $U^{n-1}_+=\lim_{t \downarrow t_{n-1}}U|_{I_n}(t)$. We also define $[U]^n = U^n_+ - U^n $ for $n=1,\dots, N-1$. We define $U'$ by taking the derivative on each interval $I_n$:\[
U' = \sum_{n=1}^N U|_{I_n}'\chi_{I_n},
\]
with $\chi_{I_n}$ the characteristic function for $I_n$.
We define the following bilinear form on $H^{1,b}(I; \hat\cH)$ which corresponds to a discrete time derivative:
\[
(Y,X)\mapsto D_{\hat\cH}(Y,X) := \sum_{n=1}^N \int_{I_n} ( Y' , X)_{\hat\cH} + \sum_{n=1}^{N-1} ( [Y]^n, X_+^n)_{\hat\cH} +( Y^0_+, X_+^0)_{\hat\cH} 
\]
and we define 
\[
(Y,X)\mapsto D_{\hat\cH}^*(Y,X) := \sum_{n=1}^N \int_{I_n} ( Y , X')_{\hat\cH} + \sum_{n=1}^{N-1} ( Y^n, [X]^n)_{\hat\cH}- ( Y^N, X^N)_{\hat\cH} .
\]
We will need some basic properties for $D_{\hat\cH}$ and $D_{\hat\cH}^*$. These results are standard and can be found in \cite[Chapter 12]{Thomee}.
\begin{lemma}\label{antiSym}  The following holds for any Hilbert space $\hat \cH$:
\begin{align}
D_{\hat \cH}(u,v) & = -D_{\hat \cH}^*(u,v) \quad \text{for all}~~u,v\in H^{1,b}(I;\hat \cH) \label{partialDM}\\
 D_{\hat \cH}(u,u) & \geq \frac12  \| u^{N}\|^2_{\hat \cH} \quad \text{for all}~~u\in H^{1,b}(I;\hat \cH).
\end{align}
\end{lemma}
\begin{proof} 
Equation \eqref{partialDM} follows by integration by parts on each interval $I_n$. The second part follows from
\begin{equation*} \label{partIntTheta}
 \begin{split} 2D_{\hat\cH}(u,u) 
 =& \sum_{n=1}^N 2 \int_{I_n} ( u' , u)_{\hat\cH} + \sum_{n=1}^{N-1} 2 (  [u]^n, u_+^n )_{\hat\cH}  + 2( u^0_+ , u_+^0)_{\hat\cH}
\\ =&
\sum_{n=1}^{N} (\| u^n\|_{\hat\cH}^2 -  \| u^{n-1}_+\|_{\hat\cH}^2 ) + \sum_{n=1}^{N-1} 2 ( [u]^n, u_+^n )_{\hat\cH}  +2 \| u^{0 }_+\|_{\hat\cH}^2
\\ = &
\|u^N\|_{\hat\cH}^2+\|u^0_+\|_{\hat\cH}^2+ \sum_{n=1}^{N-1}\left( \| u^{n}\|_{\hat\cH}^2 -  \| u^{n}_+\|_{\hat\cH}^2 + 2 \| u^{n}_+\|_{\hat\cH}^2 - 2(u^n, u^n_+  )_{\hat\cH}\right) 
\\ = &
\| u^{N}\|_{\hat\cH}^2 + \| u^{0}_+\|_{\hat\cH}^2 + \sum_{n=0}^{N-1}\left( \| u^{n}\|_{\hat\cH}^2 + \| u^{n}_+\|_{\hat\cH}^2  - 2(u^n, u^n_+  )_{\hat\cH} \right)
\\ \geq &
\| u^{N}\|_{\hat\cH}^2 ,
\end{split}
\end{equation*}
where we have used $(u^n- u^n_+ ,u^n- u^n_+ )_{\hat\cH}\geq 0$ in the last step. 
\end{proof}
\ \\[1ex]
We now recall a projection operator that plays a key role in the error analyses presented in \cite{Thomee,Schwab}. This operator is \emph{not} used { explicitly} in a standard DG time discretization method applied to a parabolic problem without constraints. It will,  however, be used in the DG method for the discretization of the problem with constraints \eqref{problem}, that is introduced below.  

\begin{mydef} \label{defI} (as in \cite{Schwab})  {Let $J=(a,b)$. For a scalar function $\phi \in L^2(J)$ which is continuous at $t=b$ we define  its projection $\Pi_J^{q-1} \phi \in \cP_{q-1} (J)$ by the $q$ conditions
\[
  \Pi_J^{q-1}\phi(b)=\phi(b), \quad \int_a^b (\Pi_J^{q-1} \phi - \phi)\cdot\psi  = 0 \quad \forall~ \psi \in \cP_{q-2}(J).
\]
(for $q=1$ only the first condition is used). 
}

For $\phi \in L_c^2(I):=\{\, \psi\in L^2(I)~|~\psi_{|I_n}~~\text{is continuous at}~t_n,~1 \leq n \leq N\,\}$ we define  the projection $\cI_{\cP^b}:  L_c^2(I) \to \cP^b(I)$ by $(\cI_{\cP^b} \phi)_{|I_n}:=\Pi_{I_n}^{q-1}(\phi_{|{I_n}})$.
  Finally for any Hilbert space $\hat \cH$ and $u\in L^2_c(I;\hat \cH)= L_c^2(I)\otimes \hat \cH $, we define $\cI_q u := (\cI_{\cP^b} \otimes {\rm id}_{\hat\cH}) u$. 
\end{mydef}
 \ \\[1ex]
For the derivation of properties of this projection operator we refer to \cite{Schwab,Thomee}.
A useful characterization of the projection is given in the following lemma.
\begin{lemma}\cite[p. 207-208]{Thomee} \label{LemmaProj}
Let $w\in H^{1,b}(I;\hat\cH)$ for any separable  Hilbert space $\hat \cH$. {The solution $W \in \PP^b(I; \hat\cH)$ of 
\begin{eqnarray}\label{easyOdeDisc}
D_{\hat \cH}(W,X)  &=& D_{\hat \cH}(w,X) \quad \text{for all} ~~X\in \PP^b(I; \hat \cH)
\end{eqnarray}
fulfills $W=\cI_{q} w $.}
\end{lemma}
\begin{proof}
By taking an orthogonal basis of $\hat \cH$ it suffices to establish this for the scalar version of \eqref{easyOdeDisc}, which  reads as follows: for given $f \in H^{1,b}(I)$ find $F\in \PP^b(I)$ such that for all $n=1,\dots, N$:
\begin{eqnarray}
\int_{I_n} (F'-f')\cdot 1 + (F-f)^{n-1}_+ - (F-f)^{n-1} &=& 0 \quad\text{with}~~(F-f)^0:=0 \label{ScalProjQ1} \\
\text{and}~~\int_{I_n} (F'-f')(t-t_{n-1})^{k+1}  &=&0 \label{ScalProjQ2}, \quad \text{for}~~k= 0,\dots, q-2.
\end{eqnarray}
After integration \eqref{ScalProjQ1} becomes \[
 (F - f)^n = (F - f)^{n-1} \quad\text{with}~~(F-f)^0:=0.
\]
From this we obtain $F(t_n) = f(t_n)$ for all $n=1,\dots,N$ by induction.
Integration by parts in \eqref{ScalProjQ2}, together with $F(t_n) = f(t_n)$ shows that \eqref{ScalProjQ2} is equivalent to \[
\int_{I_n} (F-f)\cdot p  =0 \quad \text{for all } p\in \cP_{q-2}(I_n).
\]
Hence $F=\cI_{\cP^b} f$, cf. Definition~\ref{defI}.
\end{proof}
\ \\[1ex]
We now describe the DG time discretizations of the constrained and mixed formulations \eqref{problem} and \eqref{problemmixed}, respectively.
We introduce the notation
\[
 K(u,v) :=D_\cH (u,v) + \int_I a(u, v),  \qquad u,v \in H^{1,b}(I;\cH) \cap L^2(I;\cU). 
\]
{\bf Time-discrete constrained formulation.}
The discrete (in time) version of \eqref{problem}  reads: find $U \in \PP^b(I;\cU)$ such that
\begin{align} 
K(U, X) & = \int_I f(X) + (u_0,X_+^0)_\cH \quad \text{for all}~~X \in \PP^b(I;\cV) \label{Divp} \\
\int_I b(U,R) & = \int_I ( \cI_q g)(R)    \quad \text{for all}~~R \in \PP^b(I;\cQ).  \label{Divp1}
\end{align}
{\bf Time-discrete mixed formulation.}
The discrete (in time) version of \eqref{problemmixed}  reads: find $U \in \PP^b(I;\cU)$, $P\in \PP^b(I;\cQ_1)$ such that
\begin{align} 
K(U, X) + \int_I b_1(X,P)& = \int_I f(X) + (u_0,X_+^0)_\cH \quad \text{for all}~~X \in \PP^b(I;\cV_2) \label{Divpmixed} \\
\int_I b(U,R) & = \int_I ( \cI_q g)(R)    \quad \text{for all}~~R \in \PP^b(I;\cQ).  \label{Divp1mixed}
\end{align}
\emph{Note that in  the discretization of the constraints the projection operator $\cI_q$ is used}, cf. Remark~\ref{remprojection} below. If $B$ corresponds to a trace operator (boundary condition) this is the usual treatment of an inhomogeneous Dirichlet boundary {condition
, }\emph{but with the Dirichlet data projected in a suitable manner in time by means of} $\cI_q$. In case of a Stokes problem with a homogeneous Dirichlet boundary condition, the condition in \eqref{Divp1mixed} is the divergence constraint with \emph{projected data} $\cI_q g$. 
\begin{remark} \label{remprojection}
 \rm The formulations of the constraint equations in \eqref{Divp1} and \eqref{Divp1mixed} are given in variational form because this is closest to the actual inplementation of the method, cf. Subsection~\ref{sectdiscrete}. For the analysis below it is useful to introduce an equivalent operator formulation. In operator form the equations \eqref{Divp1} and \eqref{Divp1mixed} are equivalent to $BU=\cI_q g$ in $\PP^b(I;\cQ')$. Due to the tensor product structure of $\cI_q$ and the facts that $B$ does not depend on $t$ and $U\in  \PP^b(I;\cU)$ we get $BU=B\cI_qU=\cI_qBU$. Hence the constraint conditions \eqref{Divp1} and \eqref{Divp1mixed} have the equivalent operator formulation
 \[
  \cI_q(BU-g)=0 \quad  \text{in}~~ \PP^b(I;\cQ').
 \]
In the error analysis we use this more compact representation of the discrete constraints. In the mixed formulation,  for the case $b_2 \neq 0$, cf. Example~\ref{example3}, it is natural to split the constraint equation \eqref{Divp1mixed} in  two parts:
\begin{align}
 \int_I b_1(U,R)= b(U,(R,0)) & = \int_I ( \cI_q g_1)(R)    \quad \text{for all}~~R \in \PP^b(I;\cQ_1), \label{split1A}\\
\int_I b_2(U,R)= b(U,(0,R)) & = \int_I ( \cI_q g_2)(R)    \quad \text{for all}~~R \in \PP^b(I;\cQ_2), \label{split2A}
\end{align}
where $g=(g_1,g_2)$.
Note that in \eqref{split1A} the same bilinear form $b_1(\cdot,\cdot)$ and Lagrange multiplier space $\PP^b(I;\cQ_1)$ as in \eqref{Divpmixed} are used. 
\end{remark}
\ \\  
In the following lemma we derive consistency and well-posedness of these discretizations.  
\begin{lemma}\label{LemmaDivp}
Assume that the solution $u$ of \eqref{problem} has regularity $u \in H^1(I;\cH)$ and assume that $g\in L^2_c(I; \cQ')$. The time-discrete constrained and mixed formulations  \eqref{Divp}-\eqref{Divp1}  and \eqref{Divpmixed}-\eqref{Divp1mixed} are consistent, i.e.: 
\begin{align} \label{consis1}
K(u,X)& = \int_I f(X) + (u_0,X_+^0)_\cH \quad \text{for all}~~X\in \PP^b(I;\cV) \\
  \cI_q(Bu-g) & = 0 , \label{consis2}
\end{align}
and
\begin{align} \label{consis1mixed}
K(u, X) + \int_I b_1(X,p)& = \int_I f(X) + (u_0,X_+^0)_\cH \quad \text{for all}~~X \in \PP^b(I;\cV_2) \\
  \cI_q(Bu-g) & = 0  \label{consis2mixed}
\end{align}
hold. 
For $f\in L^2(I;\cV')$ the time-discrete  problem \eqref{Divp}-\eqref{Divp1} has a unique solution $U$. If $f$ has regularity $f\in L^2(I;\cV_2')$ then there is a unique $P\in  \PP^b(I;\cQ_1)$ such that with the unique solution $U$ of \eqref{Divp}-\eqref{Divp1} the pair $(U,P)$ is the unique solution of the time-discrete mixed  problem \eqref{Divpmixed}-\eqref{Divp1mixed}.
\end{lemma}
\begin{proof}
The consistency properties are evident (note that due to continuity in time of $u$ the jump terms in $D_\cH(\cdot,\cdot)$ vanish). Existence of a discrete solution of \eqref{Divp}-\eqref{Divp1} is proved by a standard shift argument as follows. Using that  $B:\cU\rightarrow \cQ'$ is surjective, we find a $v \in L^2_c(I;\cU)$ which satisfies the constraint $Bv=g$. Take $V:=\cI_q v \in \PP^b(I;\cU)$. Due to ellipticity of $a(\cdot,\cdot)$ on $\cV$ and $D_\cH(X,X) \geq 0$ for all $X \in \PP^b(I;\cU)$ the variational problem: find $W \in \PP^b(I;\cV)$ such that
\[
  K(W,X)= \int_I f(X) + (u_0,X_+^0)_\cH - K(V,X) \quad \text{for all}~~X\in \PP^b(I;\cV) 
\]
has a unique solution. Take $U:=W+V \in  \PP^b(I;\cU)$. Then $U$ satisfies \eqref{Divp} and we have
\[
  \cI_q(BU-g)=\cI_q(BV-g)=\cI_q(B\cI_q v-g)= \cI_q(\cI_q B v-g)= \cI_q(Bv-g)=0,
\]
hence also the discrete constraint is satisfied. Suppose that there are two solutions $U_1, U_2$ of \eqref{Divp}-\eqref{Divp1}. Then $BU_i=B(\cI_q U_i)=\cI_q BU_i= \cI_q g$ holds, hence $U_1-U_2 \in L^2(I;\cV)$. From the ellipticity of $K(\cdot,\cdot)$ on $\PP^b(I;\cV)$ we get $U_1=U_2$.

It remains to prove that \eqref{Divpmixed}-\eqref{Divp1mixed} has a unique solution. If $(U,P)$ solves \eqref{Divpmixed}-\eqref{Divp1mixed}, then $U$ solves \eqref{Divp}-\eqref{Divp1}. Conversely, if $U$ solves \eqref{Divp}-\eqref{Divp1}, then $(U,P)$ solves \eqref{Divpmixed}-\eqref{Divp1mixed} if and only if 
\[
B'_1 P (X) =   \int_I f(X) + (u_0,X_+^0)_\cH - K(U,X) \quad \text{for all}~~X\in \PP^b(I;\cV_2) 
\]
holds. Since the right hand side vanishes for all $X\in \PP^b(I;\cV)$ and $B'_1:\cQ_1 \mapsto \cV_2^0$ is an isomorphism, there exists a unique $P\in \PP^b(I;\cQ_1)$ which satisfies this relationship.
\end{proof}
\ \\
\begin{remark} \rm In the analysis below we will see that the projection operator $\cI_q$ used in the time-discrete constraints \eqref{Divp1} and \eqref{Divp1mixed} is essential for optimal order convergence results. Applying Lemma \ref{LemmaProj} to \eqref{Divp1} (or \eqref{Divp1mixed}) shows that \eqref{Divp1} is equivalent to $D_{\cQ'} BU = D_{\cQ'} g$, which is the discrete  analogon of $Bu'=g'$ with $Bu(0)=g(0)$. The discretization of the constraint by $D_{\cQ'} BU = D_{\cQ'} g$ is very similar to the index reduction approach used  in  \cite{AltmannZimmer} for obtaining accurate Runge-Kutta discretizations of DAEs with differentiation index 2 (such as the Stokes equation),  cf. equation (3.3c) in \cite{AltmannZimmer}. {The projection operator $\cI_q$ is also used in \cite{Walkington2014} to obtain optimal error bounds for the wave equation with non-homogeneous boundary conditions.}
 \end{remark}

\section{Error analysis for time-discrete constrained formulation} \label{SectionErrorAnalysis}
In this section we derive optimal discretization error bounds for the discrete constrained formulation \eqref{Divp}-\eqref{Divp1}. We first derive global bounds in the energy norm and then give a superconvergence result.
\subsection{Optimal global error bounds} \label{sectglobal}
We apply a standard argument as in \cite{Schwab} and show that the discretization error can be bounded by the projection error. For this projection there are error bounds available.
\begin{theorem} \label{thmdiscr1}
Assume that the solution $u$ of \eqref{problem} has regularity $u \in L^2_c(I;\cU)\cap H^1(I;\cH)$ and let $U\in\PP^b(I;\cU)$ be the solution of \eqref{Divp}-\eqref{Divp1}. The following holds:
\begin{equation} \label{Ceabound}
 \|u(T)-U(T)\|_\cH + \|u-U\|_{L^2(I;\cU)} \leq (1+\sqrt{2} c_\gamma \Gamma)\|u-\cI_q u\|_{L^2(I;\cU)},
\end{equation}
with $c_\gamma:= \max \{\frac{1}{\gamma}, 1\}$.
\end{theorem}
\begin{proof}
 Define $E:=U-\cI_q u$. Note that
\begin{equation} \label{crucial}
\begin{split}
  BE & =BU-B(\cI_q u)=B(\cI_q U)-B(\cI_q u) \\ & = \cI_q(BU-Bu)=\cI_q(BU-g)=0.
\end{split}
\end{equation}
Hence $E \in \PP^b(I;\cV)$ can be used as a test function in \eqref{Divp}. This then yields, using the ellipticity of $a(\cdot,\cdot)$ on $\cV$ and $D_\cH(E,E) \geq \frac12 \|E(T)\|_\cH^2$, with $c_\gamma = \max \{\frac{1}{\gamma}, 1\}$:
\begin{align*}
\frac12 \|E(T)\|_\cH^2 +  \|E\|_{L^2(I;\cU)}^2 &  \leq c_\gamma K(E,E) = c_\gamma K(u-\cI_q u, E) \\
 & =c_\gamma \big( D_\cH( u-\cI_q u, E) +\int_I a( u-\cI_q u, E) \big). 
\end{align*}
Using \eqref{easyOdeDisc} we get  
\begin{equation} \label{crucialA}
D_\cH( u-\cI_q u, E)=0,
\end{equation} 
 and thus
\begin{align*}
\frac12 \|E(T)\|_\cH^2 +  \|E\|_{L^2(I;\cU)}^2 & \leq c_\gamma \Gamma \|E\|_{L^2(I;\cU)}\|u-\cI_q u\|_{L^2(I;\cU)}
\\ & \leq \frac12 c_\gamma^2 \Gamma^2\|u-\cI_q u\|_{L^2(I;\cU)}^2 +\frac12 \|E\|_{L^2(I;\cU)} ^2
\end{align*}
holds.  This yields $\|E(T)\|_\cH +  \|E\|_{L^2(I;\cU)} \leq \sqrt{2} c_\gamma \Gamma \|u-\cI_q u\|_{L^2(I;\cU)}$. Using this, a triangle inequality and $u(T)- (\cI_q u)(T)=0$ we obtain the result \eqref{Ceabound}.
\end{proof}
\ \\
\begin{remark} \rm \label{RemarkSubOptIq}
 For this analysis to work and the projection error bound \eqref{Ceabound} to hold it is crucial, cf.  \eqref{crucial} that in the discretization of the constraint we use $BU=\cI_q g$ and not $BU=g$ or some other interpolation (in time) of the data $g$. In Subsection \ref{NumOne} we will present a numerical experiment where the use of $BU=g$ instead of $BU=\cI_q g$ leads to sub-optimal results. 
\end{remark}
\ \\[1ex]
Results for the projection error $\|u-\cI_q u\|_{L^2(I;\cU)}$ are known in the literature \cite[Theorem 3.10]{Schwab}. Using these results we obtain the following optimal discretization error bound.
\begin{theorem} \label{mainthm1}
Let  $U\in\PP^b(I;\cU)$ be the solution of \eqref{Divp}-\eqref{Divp1}, and assume that the solution $u$ of \eqref{problem} has smoothness $u \in  H^m(I;\cU)$ for an $m$ with $1 \leq m \leq q$. The following holds:
\begin{align}
 \|u(T)-U(T)\|_\cH + \|u-U\|_{L^2(I;\cU)} & \leq c  \left(\sum_{n=1}^N k_n^{2m} \|u^{(m)}\|_{L^2(I_n;\cU)}^2\right)^{\frac{1}{2}} \label{CeaboundA}
\\& \leq c k^m \|u^{(m)}\|_{L^2(I;\cU)}
\end{align}
for some $c>0$ which only depends on $q,\gamma$ and $\Gamma$.
\end{theorem}

\begin{remark} 
\rm
Typically one is interested in the case $q=m$. The other cases will however also be of some use in what follows. For the homogeneous case $Bu=0$ the result \eqref{CeaboundA} is the same as the one in \cite[Theorem 12.1]{Thomee}.

In applications the constraint $Bu=g$ may be incompatible this the initial condition $u(0)=u_0$. This leads to low regularity of  the solution $u$ at $t=0$, i.e. $\|u^{(m)}\|_{L^2(I_n;\cU)}^2$ is unbounded for larger $m$ and small $n$. In such a setting the step size $k_n$ in \eqref{CeaboundA} needs to be taken sufficiently small in the initial steps. 
In the specific case of Example \ref{Exheat} and Example \ref{example3} this is the case if the Dirichlet boundary condition is incompatible with the initial condition, cf. \cite{Schwab}.
\end{remark}
\ \\[1ex]
We also derive a bound for the error in the time derivative, which will be used in the analysis of superconvergence.
\begin{theorem} 
Assume that the solution $u$ of \eqref{problem} has regularity $u \in H^1(I;\cH)$ and assume that $g\in L^2_c(I; \cQ')$. Let  $U\in\PP^b(I;\cU)$ be the solution of \eqref{Divp}-\eqref{Divp1}. The following holds:
\begin{equation}
\|u' - U'\|_{L^2(I;\cV')}  \leq c   \left(\|u'-(\cI_q u)'\|_{L^2(I;\cV')} + \|u-U\|_{L^2(I;\cU)} \right) \label{resA}
\end{equation}
for some $c>0$ which only depends on $q$ and $\Gamma$.
If we additionally assume that $u$ has smoothness $u \in  H^{m-1}(I;\cU)\cap H^{m}(I;\cV')$ for an $m$ with $2 \leq m \leq q$, then the following holds:
\begin{align}
\|u'-U'\|_{L^2(I;\cV')}  &\leq c' \left(\sum_{n=1}^Nk_n^{2m-2} (\|u^{(m)}\|_{L^2(I_n;\cV')}^2+\|u^{(m-1)}\|_{L^2(I_n;\cU)}^2) )\right)^{\frac{1}{2}}\nonumber
\\ & \leq c' k^{m-1} (\|u^{(m)}\|_{L^2(I;\cV')}+\|u^{(m-1)}\|_{L^2(I;\cU)}),\label{CeaboundB}
\end{align}
for some $c'>0$ which only depends on $q, \gamma$ and $\Gamma$.
\end{theorem}
\begin{proof}
We will use $E = U - \cI_q u \in \PP^b(I;\cH)$. Using the Riesz representation theorem in $\cV$ (with $\|\cdot\|_\cV=\|\cdot\|_\cU$), we find $\tilde E \in  \oplus_{n=1}^N \PP_{q-2}(I_n; \cV) $ such that \[
\|E'(t)\|^2_{\cV'} = (E'(t),\tilde{E}(t))_\cH = \|\tilde E(t)\|^2_{\cU}\quad\text{for almost all } t\in I.
\]
We define $X=\sum_{n=1}^N k_n^{-1} (t-t_{n-1}) \tilde E(t) \chi_{I_n} \in \PP^b(I;\cV)$ and using Lemma \ref{LemmaProj} and the Galerkin orthogonality, cf. Lemma~\ref{LemmaDivp}, we get
\[ \begin{split}
D_\cH(E,X)  & = D_\cH(U - \cI_q u,X)=D_\cH(U - u,X) \\ & =K(U- u, X)-\int_I a(U-u,X)= \int_I a(u-U,X).
\end{split} \]
This yields 
\begin{align*}
D_\cH(E,X) 
& \leq \sum_{n=1}^N \int_{I_n} k_n^{-1} (t-t_{n-1}) \Gamma \|u-U\|_\cU \|\tilde E \|_{\cU}  
\\& \leq \frac{1}{2}\sum_{n=1}^N  \int_{I_n} k_n^{-1}(t-t_{n-1}) \left( \Gamma^2 \|u-U\|_\cU ^2+ \|E'\|_{\cV'}^2 \right).
\end{align*}
Noting that the left-hand side is \[
D_\cH(E,X) = \sum_{n=1}^N \int_{I_n} k_n^{-1}(t-t_{n-1}) (E',\tilde E)_\cH = 
\sum_{n=1}^N \int_{I_n} k_n^{-1}(t-t_{n-1}) \|E'\|_{\cV'}^2 ,
\]
we get 
\[
\sum_{n=1}^N \int_{I_n} k_n^{-1} (t-t_{n-1}) \|E'\|_{\cV'}^2 \leq \sum_{n=1}^N  \int_{I_n} k_n^{-1}(t-t_{n-1})  \Gamma^2 \|u-U\|_{\cU}^2 \leq \Gamma^2 \|U-u \|_{L^2(I;\cU)}^2.
\]
Using that the $L^1$-norm on $\PP_{2q-4}([0,1])$ is equivalent to the norm $\phi \mapsto \int_0^1 t|\phi(t)| dt$, we get 
 \[
\int_I \|E'\|_{\cV'}^2 = \sum_{n=1}^N \int_{I_n} \|E' \|^2_{\cV'} \leq \sum_{n=1}^N \int_{I_n} C k_n^{-1} (t-t_{n-1}) \|E'\|_{\cV'}^2 \leq C \Gamma^2 \|U-u \|_{L^2(I;\cU)}^2,
\]
for some constant $C>0$ which depends only on $q$. Using a triangle inequality completes the proof of \eqref{resA}. 
Results for the projection error $\|u'-(\cI_q u)'\|_{L^2(I;\cV')}$ are known in the literature  \cite[p.~214]{Thomee}. Using these results and Theorem \ref{thmdiscr1} we obtain the optimal discretization error bound \eqref{CeaboundB}.
\end{proof}

\subsection{Optimal superconvergence result}\label{SubSectionOptSuper}
As is known from the literature (see \cite[Theorem 12.3]{Thomee}), for the case with a \emph{homogeneous} constraint in \eqref{problem}  we get a superconvergence result provided the solution $u$ is sufficiently smooth.  The analysis in \cite[Theorem 12.3]{Thomee} is not directly applicable to the case with an inhomogeneous constraint. Note that in the latter case the projection operator $\cI_q$ is used in the discretization method, cf.~\eqref{Divp1}. In this subsection we derive an optimal superconvergence result for the discretization \eqref{Divp}-\eqref{Divp1}. For this we first introduce an abstract notion of regularity, which is very similar to  the one used in \cite{Thomee}. 
We define $\cV^0 := \overline{\cV}^{\cH}$ (with norm $\|\cdot\|_\cH$),  $\cV^1 := \cV$ (with norm $\|\cdot\|_\cU $), which induces the Gelfand triple $\cV \hookrightarrow \cV^0 \cong (\cV^0)' \hookrightarrow \cV'$. Furthermore  \[
\cV^2 := \{\,u\in \cV~|~Au \in (\cV^0)' \, \}.
\]
Based on  the Riesz representation theorem in $\cV^0$, for any $u\in \cU$ with $Au \in (\cV^0)'$ we define $\tilde L u\in \cV^0$ by $(\tilde Lu,v)_\cH= a(u,v)$ for all $v \in \cV$. Let us consider the restriction of $L:=\tilde L|_{\cV^2}$, i.e., $L:\,\cV^2 \to \cV^0$ with
\begin{equation}\label{propLt}
  (Lu,v)_\cH= (Au)(v)=a(u,v) \quad \text{for all}~~u \in \cV^2,\, v \in \cV.
\end{equation}

For $u \in \cV^2$ with $Lu=0$ it follows that $a(u,v)=0$ for all $v \in  \cV$, hence, $u=0$. Take $f \in \cV^0$. There exists a unique $u \in \cV$ such that $a(u,v)=(f,v)_\cH$ for all $v \in  \cV$. This implies $\|Au\|_{(\cV^0)'} = \|f\|_\cH$, hence $u \in \cV^2$. We conclude that $L:\,\cV^2 \to \cV^0$ is a bijection. Its inverse $L^{-1}:\, \cV^0 \to \cV^2 \subset \cV^0$ is symmetric, bounded and positive. The unique square root operator of $L^{-1}$ (that commutes with $L^{-1}$) is denoted by $L^{-\frac12}:\cV^0 \rightarrow \cV^0$. We define $L^{\frac{1}{2}} := L^{-\frac{1}{2}} L:\cV^2 \rightarrow \cV^0$, hence, $(L^{\frac{1}{2}}v, L^{\frac{1}{2}}v )_\cH=(Lv,v)_\cH= a(v,v)$ for all $v \in \cV^2$. The definition of $L^\frac12$ can be extended to $\cV$. First note that $A:\cV \rightarrow \cV'$ is a homeomorphism and $\cV^2$ is defined as the preimage of $(\cV^0)'$. Since $(\cV^0)'\subset \cV'$ is dense, it follows that $\cV^2 \subset \cV$ is dense. 
Using the ellipticity and continuity of $a(\cdot,\cdot)$ we see that 
\[
\gamma^{\frac12} \|v\|_{\cU} \leq  a(v,v)^{\frac{1}{2}} =\|L^{\frac{1}{2}}v\|_\cH \leq \Gamma^{\frac12} \|v\|_{\cU} \quad \text{for all}~~v \in \cV^2.
\]
Thus, using the density of $\cV^2$ in $\cV$, we conclude that $L^{\frac12}:\, \cV \to \cV^0$ is a homeomorphism. Since $L^{\frac{1}{2}}$ is symmetric (due to $L^{-\frac{1}{2}} L=LL^{-\frac{1}{2}}$), the dual norm equivalence reads  $\gamma^{\frac12} \|v\|_{\cH} \leq \|L^\frac12 v\|_{\cV'}\leq \Gamma^\frac12 \|v\|_{\cH}$ for all $v \in  \cV$. 

For $m \in \Bbb{N},~m> 2$, we define  
\[
\cV^m = \{\, u\in \cV^{2}~|~ L u \in \cV^{m-2}\, \}.
\]
Note that $L^\frac12 L  v=L L^\frac 12 $ for all $v \in \cV^3$ and $\cD(L^\frac{m}{2})=\cV^{m}$, $m \in \Bbb{N}$. As a measure for regularity we define
\[
|u|_m:=\sup_{v\in \cV^m} \frac{(u,L^{\frac{m}{2}} v)_\cH}{\|v\|_\cH}, \quad u \in \cH,\quad m \in\Bbb{N}
\]
(note that $|u|_m = \infty$ is allowed). 

Below we use the notation $L^\frac{m}{2} u := ({\rm id}_{L^2(I)} \otimes L^\frac{m}{2})u$, $\tilde L u= ({\rm id}_{L^2(I)} \otimes \tilde L)u$, for $m\in \mathbb{Z}$ and for functions $u$ for which the right hand-side expressions are well-defined.
\ \\
For the analysis we introduce a standard dual \emph{homogeneous} problem, i.e., source term $f=0$ and with a homogeneous constraint $Bu=0$;  due to duality $t$ is replaced by $T-t$ and we prescribe an initial value at $t=T$. This leads to the following problem formulation:
find $z \in W^1(\cV;\cV')$ with $z(T)=\phi \in \cV^0$ and
\begin{equation}
- z'(t) + A z(t) = 0 \quad \text{in}~~L^2(I;\cV').
\label{ivpTranspose}
\end{equation}
This problem is well posed.
In the analysis below we need  that the solution $z$ has regularity $z \in H^1(I;\cH)$. It is easy to see that $z$ has this regularity property if $\phi \in \cV$. Indeed, for $\phi \in \cV$ let $y\in {W^1(\cV;\cV')}$ be the solution of the parabolic problem \eqref{ivpTranspose} with $y(T)=L^{\frac{1}{2}} \phi \in \cV^0$, then $z= L^{-\frac{1}{2}} y$ solves \eqref{ivpTranspose} with $z(T)=\phi$ and $z'=L^{-\frac{1}{2}} y' \in L^2(I;\cH) $. Hence, for $\phi \in \cV$ we have $z \in H^1(I;\cH)$. 
This implies that \eqref{ivpTranspose} holds in $L^2(I,(\cV^0)')$. 

The corresponding discrete problem is as follows: determine $Z \in \PP^b(I;\cV)$ such that 
\begin{equation}\label{DivpTranspose}
K(X,Z) = -D^*_\cH(X,Z)  + \int_I a(Z, X)  = (\phi,X^N)_\cH \quad \text{for all}~~X\in \PP^b(I;\cV).
\end{equation}
This is the analogue of \eqref{Divp} for \eqref{ivpTranspose}.
This \emph{homogeneous} backward problem has been discussed in detail in \cite[p. 212-213]{Thomee}. If we replace $t$ by $T-t$, then we can apply the results from the previous subsection to $z$ and $Z$. In particular we can derive the following error estimate.
\begin{lemma} \label{lemmadualP}
Let $z$ and $Z$ be the solution of \eqref{ivpTranspose} and \eqref{DivpTranspose}, respectively and assume that $z\in H^1(I;\cH)$. The following error estimate holds for $1 \leq m \leq q$:
\begin{equation} \label{errordual}
 \|L^{-m+1+\frac12} (z-Z)\|_{L^2(I;\cH)}+ \|L^{-m + \frac12} (z'-Z')\|_{L^2(I;\cH)} \leq C k^{m-1} \|\phi\|_\cH
\end{equation}
for some $C>0$ which only depends on $q,\gamma$ and $\Gamma$.
\end{lemma}
\begin{proof}
First consider the case $m=1$. 
 The homogeneous problem \eqref{ivpTranspose} can be formulated as $z'= Lz$ in $L^2(I;\cV^0)$. From $(z',z)_{L^2(I;\cH)}=(L^{\frac{1}{2}}z,L^{\frac{1}{2}} z)_{L^2(I;\cH)}$ and integration by parts, we get $\|L^\frac12 z\|_{L^2(I;\cH)}^2 \leq \frac{1}{2} \|\phi\|_{\cH}^2$. Similarly, taking $X=Z$ in \eqref{DivpTranspose} and applying Lemma \ref{antiSym} shows that 
\[
\frac{1}{2}\|Z^N\|^2_\cH + \|L^{\frac{1}{2}} Z \|_{L^2(I;\cH)}^2 \leq  -D^*_\cH(Z,Z)  + \int_I a(Z, Z) \leq \frac{1}{2} \|\phi\|_\cH^2+\frac{1}{2} \|Z^N\|_\cH^2.
\]
We conclude that 
\begin{equation}\label{contOfHomBoth}
 \|L^{\frac{1}{2}} z \|_{L^2(I;\cH)}^2 +\|L^{\frac{1}{2}} Z \|_{L^2(I;\cH)}^2 \leq \|\phi\|_{\cH}^2.
\end{equation}
From \eqref{resA} and the standard bound $\|z'-(\cI_q z)'\|_{L^2(I;\cV')}\leq \tilde c \|z'\|_{L^2(I;\cV')}$ (for some $\tilde c$ which depends only on $q$, cf. \cite[p.~214]{Thomee}), we know that   
\begin{align*}
\|L^{-\frac12} (z'-Z')\|_{L^2(I;\cH)}  & \leq  \gamma^{-\frac{1}{2}} \|z'-Z'\|_{L^2(I;\cV')} \\
  & \leq c \big( \|z'\|_{L^2(I;\cV')} + \|z\|_{L^2(I;\cU)} +\|Z\|_{L^2(I;\cU)} \big) \\
 &  \leq c \big(\|z'\|_{L^2(I;\cV')}+   \gamma^{-\frac{1}{2}}  \|L^{\frac{1}{2}} z\|_{L^2(I;\cH)}  + \gamma^{-\frac{1}{2}} \|L^{\frac{1}{2}} Z\|_{L^2(I;\cH)}),
\end{align*}
with a constant $c$ which only depends on $q,\gamma$ and $\Gamma$.
Noting that $\|z'\|_{L^2(I;\cV')} = \|Lz\|_{L^2(I;\cV')} \leq \Gamma^{\frac{1}{2}} \|L^{\frac{1}{2}}z\|_{L^2(I;\cH)}$ and using \eqref{contOfHomBoth}, completes the proof for $m=1$.
We now consider the case $m\geq 2$. Let $y=L^{1-m} z$, $Y = L^{1-m} Z$ and $\psi = L^{1-m} \phi$.
 Note that the solution operators  $\phi \to z$ and $\phi \to Z$ in \eqref{ivpTranspose} and \eqref{DivpTranspose} commute with $L$. Thus $y$ solves \eqref{ivpTranspose} with initial condition $y(T)=\psi$ and $Y$ solves the corresponding discrete problem 
 \begin{equation*}
K(X,Y) = -D^*_\cH(X,Y)  + \int_I a(Y, X) = (\psi,X^N)_\cH \quad \text{for all}~~X\in \PP^b(I;\cV).
\end{equation*}  
Note that   $y^{(m-1)}=L^{m-1} y =  z\in  W^1(\cV;\V')$. Using the estimates \eqref{CeaboundA}, \eqref{contOfHomBoth} and $y^{(m-1)} =  z$  we get
\[
  \|y-Y\|_{L^2(I;\cU)} \leq c k^{m-1} \|y^{(m-1)}\|_{L^2(I;\cU)}\leq   c k^{m-1} \gamma^{-\frac{1}{2}} \|L^{\frac{1}{2}} z\|_{L^2(I;\cH)} \leq c  \gamma^{-\frac{1}{2}} k^{m-1} \|\phi\|_\cH
\]
and
\[
 \|L^{-m+1 +\frac12}(z-Z)\|_{L^2(I;\cH)}= \|L^\frac12(y-Y)\|_{L^2(I;\cH)} \leq c k^{m-1} \Gamma^{\frac{1}{2}} \gamma^{-\frac{1}{2}} \| \phi\|_\cH,
\]
with a constant $c$ which only depends on $q,\gamma$ and $\Gamma$.
From this we obtain the estimate for the first term in \eqref{errordual}.
In order to bound the second term we use \eqref{CeaboundB}, \eqref{contOfHomBoth} and $y^{(m-1)} =  z$:
\begin{align*}
\|L^{-m+\frac12}(z'-Z')\|_{L^2(I;\cH)} &= \|L^{-\frac12} (y'-Y')\|_{L^2(I;\cH)} \leq \gamma^{-\frac{1}{2}} \|y'-Y'\|_{L^2(I;\cV')} \\
 & \leq  c k^{m-1}(\|y^{(m)}\|_{L^2(I;\cV')} + \|y^{(m-1)}\|_{L^2(I;\cU)}) \\
 & \leq c k^{m-1}(\Gamma^{\frac12} \|L^{-\frac12} z' \|_{L^2(I;\cH)} + \gamma^{-\frac12} \|L^{\frac{1}{2}} z \|_{L^2(I;\cH)})  \\
 & \leq   c k^{m-1} (\Gamma^{\frac12} + \gamma^{-\frac12} )  \|L^{\frac{1}{2}}z \|_{L^2(I;\cH)} \\&\leq c  k^{m-1} (\Gamma^{\frac12} + \gamma^{-\frac12} ) \| \phi\|_\cH,
\end{align*}
with a constant $c$ which only depends on $q,\gamma$ and $\Gamma$.
This yields the bound for the second term in \eqref{errordual}.
\end{proof}
\ \\[1ex]
Using this we can derive the following result,  which has a superconvergence result as an easy corollary. 
\begin{theorem} \label{thmSuper} 
Assume that the solution $u$ of \eqref{problem} has regularity $u \in L^2_c(I;\cU)\cap H^1(I;\cH)$  and $Au \in L^2_c(I;(\cV^0)')$. Let $U\in\PP^b(I;\cU)$ be the solution of \eqref{Divp}-\eqref{Divp1}. Let  $2 \leq  \ell \leq 2q-1$. 
The following holds:
 \[
\|U^N - u(T)\|_\cH \leq c k^{\frac{\ell-1}{2}} \left( \int_{I} |\tilde Lu-\cI_q \tilde Lu|_{\ell-2}^2  \right)^\frac12
\]
for some $c>0$ which only depends on $q$, $\gamma$ and $\Gamma$. 
\end{theorem}
\begin{proof}
We introduce $E:=U-u$, $\phi := E(T) \in \cV$, $\hat E = U- \cI_q u\in \PP^b(I;\cV)$. $\theta = \cI_q u - u = E-\hat E$.
We take $\phi$ in the homogeneous backward problem \eqref{ivpTranspose}. The corresponding solutions of \eqref{ivpTranspose} and \eqref{DivpTranspose} are denoted by $z$ and $Z$, respectively.
We define $\zeta = z-Z\in H^{1,b}(I;\cH)$. One  can not apply a straightforward duality argument for  the following reason. We (only) have $\cI_q B(U-u)=0$ and thus in general $E \notin L^2(I; \cV)$. Hence, we can \emph{not} use $E$ as a test function in \eqref{ivpTranspose}. To overcome this problem, below we consider a transformed dual problem with endpoint condition $L^{-1} \phi$ (instead of $\phi$) and use a test function $\tilde L E$. The precise arguments are as follows.  Due the regularity assumption $Au\in L^2_c(I;(\cV^0)')$ we have that $\tilde Lu \in L^2_c(I;\cV^0)$ is well-defined. Using the Galerkin property $K(U-u,\cdot)=0$ in $\PP^b(I,\cV')$ we have
 \[
AU = D_\cH(u-U,\cdot)+Au\quad\text{in }\PP^b(I,\cV').
\]
Due to the density of $\cV$ in $\cV^0$ it follows that $AU\in  \PP^b(I;(\cV^0)')$ and thus $\tilde LU\in \PP^b(I;\cV^0) \subset L^2_c(I;\cV^0)$ is well-defined. It follows that $\tilde LE \in L^2_c(I;\cV^0)$. Using this and that $\tilde L \theta \in L^2_c(I;\cV^0)$ is well defined, we find that $L\hat E=\tilde L \hat E = \tilde L \theta - \tilde L E$ is well defined. Therefore, we have $\hat E \in L^2_c(I; \cV^2)$. Note that due to $E(T) = \hat E(T) \in \cV^2$ we have $(\tilde L E)(T)=\tilde L E(T) = LE(T)= L \phi$.  The dual problem with endpoint condition $L^{-1} \phi$ has corresponding solutions of \eqref{ivpTranspose} and \eqref{DivpTranspose} given by $y:= L^{-1}z$ and $Y:=L^{-1}Z$. Hence,
\[
 (L^{-1} \phi,v)_\cH= K(v,y)= D_\cH(v,y)+ \int_I a(v,y) =  D_\cH(v,y)+ \int_I (v, Ly)_\cH ~ \forall~v \in L^2(I;\cV).
\]
Using the density of $\cV$ in $\cV^0$ we obtain
\[
 (L^{-1} \phi,v)_\cH=  D_\cH(v,L^{-1}z)+ \int_I (v, z)_\cH \quad \text{for all}~~v \in L^2(I;\cV^0).
\]
We use $v=\tilde L E$ as test function, which yields
\begin{equation} \label{k8}
\|\phi\|^2_\cH = (L^{-1}\phi,L\phi)_\cH = (L^{-1}\phi, \tilde L E(T))_\cH=  D_\cH(\tilde L E,L^{-1}z)+ \int_I (\tilde LE, z)_\cH.
\end{equation} 
%
Note that from Lemma \ref{LemmaProj}, the Galerkin  property $0=K(E,X)$ for all $X\in \PP^b(I;\cV)$ and $\hat E \in \PP^b(I;\cV^2)$  we get
\begin{eqnarray*}
0&=&K(E,Z)=\cD_\cH(E,Z)+\int_I a(E,Z) = \cD_\cH(\hat E,Z)+\int_I a(E,Z)
\\ &=&D_\cH(L \hat E ,L^{-1}Z)+\int_I (\tilde L E,Z)_\cH
 =D_\cH(\tilde L \hat E,L^{-1}Z)+\int_I (\tilde LE ,Z)_\cH  \\ &=& D_\cH(\tilde L E,L^{-1}Z)+\int_I (\tilde LE ,Z)_\cH.
\end{eqnarray*}
Combing  this with \eqref{k8} and using $K(\hat E,\zeta)=0$ and \eqref{partialDM}, we obtain:
\begin{eqnarray*}
\|\phi\|_\cH^2
&=& D_\cH(\tilde L E,L^{-1}\zeta)+\int_I (\tilde L E,\zeta)_\cH 
  \\ &=& D_\cH(\tilde L\theta,L^{-1}\zeta)+\int_I (\tilde L\theta,\zeta)_\cH 
       +D_\cH(\tilde L\hat E, L^{-1}\zeta)+\int_I (\tilde L\hat E,\zeta)_\cH 
\\&=& D_\cH(\tilde L\theta,L^{-1}\zeta)+\int_I (\tilde L\theta,\zeta)_\cH 
       +D_\cH( L\hat E, L^{-1}\zeta)+\int_I a( \hat E,\zeta) 
\\&=& D_\cH(\tilde L\theta,L^{-1}\zeta)+\int_I (\tilde L\theta,\zeta)_\cH+D_\cH(\hat E,\zeta)+\int_I a(\hat E,\zeta)
\\&=&D_\cH(\tilde L\theta,L^{-1}\zeta)+\int_I (\tilde L\theta,\zeta)_\cH +K(\hat E,\zeta)= -D_\cH^*(\tilde L\theta,L^{-1}\zeta)+\int_I (\tilde L\theta,\zeta)_\cH.
\end{eqnarray*}
Noting that $\theta^n = 0$ for all $n=1,\dots, N$, we get
\begin{equation}
\|\phi\|_\cH^2 = - \sum_{n=1}^N\int_{I_n} ( L^{-1}\zeta', \tilde L\theta)_\cH  +\int_I (\zeta,\tilde L\theta)_\cH .
\end{equation}
By the definition of $|\cdot|_{\ell-2}$, we have
\begin{eqnarray*}
\|\phi\|_\cH^2  
&\leq & \sum_{n=1}^N\int_{I_n} \|  L^{-\frac{\ell}{2}} \zeta' \|_{\cH} | \tilde L\theta |_{\ell-2}  +\int_I \|L^{-\frac{\ell}{2}+1}\zeta \|_{\cH} |\tilde L \theta |_{\ell-2} .
\end{eqnarray*}
Using the Cauchy-Schwarz inequality, we get 
\begin{equation}\label{BoundinBoth}
\|\phi\|_\cH^2 \leq \left(\int_I |\tilde L\theta|_{\ell-2}^2 \right)^{1/2}  \left( \| L^{-\frac{\ell}{2}} \zeta' \|_{L^2(I;\cH)} + \|L^{-\frac{\ell}{2}+1}\zeta \|_{L^2(I;\cH)} \right).
\end{equation}
We will estimate the second factor in \eqref{BoundinBoth} by using bounds from Lemma~\ref{lemmadualP}. 
First we consider the case that $\ell$ is odd. Then $m:=\frac{\ell}{2}+ \frac12$ is a natural number with $1\leq m \leq q$. We then obtain
\begin{equation} \label{tt}
\begin{split}
  & \|L^{-\frac{\ell}{2}+1}\zeta \|_{L^2(I;\cH)}+ \| L^{-\frac{\ell}{2}} \zeta' \|_{L^2(I;\cH)}\\  & =  \|L^{-(m-1)+\frac12}\zeta \|_{L^2(I;\cH)}+ \| L^{-m +\frac12} \zeta' \|_{L^2(I;\cH)} \\
 & \leq C k^{m-1} \|\phi\|_\cH = Ck^{\frac{\ell-1}{2}} \|\phi\|_\cH,
\end{split}
\end{equation}
and in combination with \eqref{BoundinBoth} this yields the desired result.
For $\ell$ even we use the following property \[
\forall x \in \cV^2: \|L^{\frac{1}{2}}x\|_\cH^2 = (Lx,x)_\cH \leq \|L x\|_\cH\|x\|_\cH.
\]
The term $ \| L^{-\frac{\ell}{2}} \zeta' \|_{L^2(I;\cH)}$ can be treated as follows, where we use  the bound derived in \eqref{tt}, Lemma \ref{lemmadualP},
\begin{align*}
  \| L^{-\frac{\ell}{2}} \zeta' \|_{L^2(I;\cH)} & \leq  \| L^{-\frac{\ell-1}{2}} \zeta' \|_{L^2(I;\cH)}^\frac12  \| L^{-\frac{\ell+1}{2}} \zeta' \|_{L^2(I;\cH)}^\frac12 \\
 & \leq C \big( k^{\frac{\ell-2}{2}}\big)^\frac12 \big( k^{\frac{\ell}{2}}\big)^\frac12 \|\phi\|_\cH = Ck^{\frac{\ell-1}{2}} \|\phi\|_\cH. 
\end{align*}
The term $\|L^{-\frac{\ell}{2}+1}\zeta \|_{L^2(I;\cH)}$ can be treated in the same way. Thus also for $\ell$ even we get the  bound as in \eqref{tt}.
\end{proof}
 \ \\

Results for the projection error $\int_I |\tilde L u-\cI_q \tilde L u|_{\ell-2}^2$ are known in the literature \cite[Theorem 3.10]{Schwab}. Using these results we obtain the following optimal discretization error bound.
\begin{theorem} \label{mainthm2}
Assume that the solution $u$ of \eqref{problem} has regularity $u \in L^2_c(I;\cU)\cap H^m(I;\cH)$ and  $Au \in H^m(I;(\cV^0)')$. Let $U\in\PP^b(I;\cU)$ be the solution of \eqref{Divp}-\eqref{Divp1}. For any $2 \leq \ell \leq 2q-1$, $1 \leq m \leq q$, we have
\begin{equation}
 \|u(T)-U(T)\|_\cH \leq c k^{\frac{\ell-1}{2}} \left(\sum_{n=1}^N k_n^{2m} \int_{I_n} |\tilde Lu^{(m)}|_{\ell-2}^2 \right)^{\frac{1}{2}}
\leq  c k^{\frac{\ell-1}{2}+m} \left(\int_I |\tilde Lu^{(m)}|_{\ell-2}^2\right)^{\frac{1}{2}} \label{CeaboundC}
\end{equation}
for some $c>0$ which depends only on $q$, $\gamma$ and $\Gamma$.
\end{theorem}
\ \\
Note that for the case of ``full regularity'', i.e, $\ell=2q-1$, $m=q$ we obtain the optimal superconvergence bound of order $k^{2q-1}$. For $\cU= \cV$ this corresponds to the result in \cite[Theorem 12.3]{Thomee}. Our estimate is more general because it states which convergence order one has, depending on the regularity of the solution, without assuming ``full regularity" $\ell=2q-1$. For the general case $\cV\neq \cU$, we have introduced the appropriate semi-norms $|\cdot|_\ell$ which express regularity. 

\section{Error analysis for time-discrete mixed formulation} \label{SectionErrorMixed}
In this section we derive discretization error bounds for the discrete  mixed  formulation \eqref{Divpmixed}-\eqref{Divp1mixed}. 
Let $(u,p)$ be the solution of \eqref{problemmixed} and $(U,P)$  the solution of \eqref{Divpmixed}-\eqref{Divp1mixed}, then $u$ and $U$ solve \eqref{problem} and \eqref{Divp}-\eqref{Divp1}, respectively, cf. Lemma~\ref{LemmaDivp}. Thus, \emph{the optimal error bounds derived in Section \ref{SectionErrorAnalysis} also hold for the solution $U$ of the discrete mixed  formulation}. It remains to bound the error for the Lagrange multiplier. We present two results. The first result gives a sub-optimal bound of order $\cO(k^{q-\frac{1}{2}})$, without making any further specific assumptions. In the second part we introduce a certain regularity assumption under which  an optimal error bound is derived.
\subsection{Sub-optimal result}\label{SubSectionBadLM}
In this subsection we derive a sub-optimal discretization error bound for the Lagrange multiplier. Below we use the notation introduced in Subsection~\ref{sectmixed}, in particular the bilinear form $b_1(\cdot,\cdot)$ and the corresponding operator $B_1: \cV_2 \to Q_1'$ with $\ker(B_1)=\cV_1$. To simplify the presentation we assume that we have a uniform step size $k_n=k$. We start with two lemmas that provide useful bounds for the error term  $\cI_q u-U$.

\begin{lemma} \label{lem89}
Let $k_n=k$ for all $n=1,\dots,N$. 
Assume that the solution $u$ of \eqref{problem} has regularity $u \in H^1(I;\cU)$. Let  $U\in\PP^b(I;\cU)$ be the solution of \eqref{Divp}-\eqref{Divp1}. The following holds:
 \[
\| (\cI_q u-U)' \|_{ L^2(I;\cH)}  \leq  c(k^\frac{1}{2} \|(u- \cI_q u)' \|_{L^2(I;\cU)} +  k^{-\frac{1}{2}}  \|\cI_q u- u \|_{L^2(I;\cU)})
\]
for some $c>0$ which depends only on $q,\gamma$ and $\Gamma$.
\end{lemma}
\begin{proof}
Let $E=\cI_q u - U\in \PP^b(I;\cV)$.
From the consistency property \eqref{consis1} and Lemma \ref{LemmaProj}  we obtain the Galerkin relation 
\[
D_\cH(E, X ) + \int_I a(u-U, X )  = 0 \quad \text{for all}~~X \in \PP^b(I;\cV).
\]
For $X:=\sum_{n=1}^N (t-t_{n-1}) \chi_{I_n} E' \in \PP^b(I;\cV)$  we get 
\begin{equation}
\int_{I} (E', X )_\cH + \int_{I} a(E, X ) + \int_{I} a(u-\cI_q u, X ) = 0 . \label{SobolevInterpolTest}
\end{equation}
The $L^1$-norm on $P^{2q-4}([0,1])$ is equivalent to
 \[
\phi \mapsto \int_0^1t|\phi(t)|\,dt.
\]
Thus, by scaling, we get the following inverse inequality \[
k \int_{I_n} \|E'\|_{\cH}^2 \leq C \int_{I_n} (t-t_{n-1}) \|E'\|_{\cH}^2
\]
for each interval $I_n$, for some $C$ which depends only on $q$. Using this and \eqref{SobolevInterpolTest} we get
\begin{equation} \label{t0}
 k \| E' \|_{ L^2(I;\cH)}^2  \leq C \int_I (E',X)_\cH = C\left(-\int_I a(E,X) -\int_I a(u-\cI_q u, X ) \right).
\end{equation}
Using integration by parts on each interval $I_n$ we obtain
\begin{equation} \label{t1}
 \begin{split}
-\int_I a(E,X)  &= -\sum_{n=1}^N \int_{I_n} a(E,(t-t_{n-1})E')  \\ & = -\sum_{n=1}^N \frac{k}{2} a(E^n,E^n) + \frac{1}{2} \int_{I} a(E,E) \leq \frac12 \Gamma \|E\|_{L^2(I;\cU)}^2.
\end{split}
\end{equation}
Integration by parts on each interval $I_n$ yields 
\begin{equation} \label{t2}
\begin{split}
& -  \int_{I} a(u-\cI_q u, X ) \leq   \left| \sum_{n=1}^N \int_{I_n} a(u-\cI_q u,(t-t_{n-1})E') \right|
\\ &= \left| \sum_{n=1}^N \int_{I_n} a((u-\cI_q u)',(t-t_{n-1})E) - \int_{I} a(u-\cI_q u,E) \right|
\\ &\leq  \sum_{n=1}^N \int_{I_n} |t-t_{n-1}| | a((u-\cI_q u)',E)| + \int_{I} |a(u-\cI_q u,E) |
\\ &\leq  k\Gamma  \|(u-\cI_q u)'\|_{L^2(I;\cU)}  \|E\|_{L^2(I;\cU)} 
+ \Gamma \|u-\cI_q u \|_{L^2(I;\cU)} \|E \|_{L^2(I;\cU)} .
\end{split}
\end{equation}
Using \eqref{t1} and \eqref{t2} in \eqref{t0} we obtain 
\[
 k \| E' \|_{ L^2(I;\cH)}^2 \leq \tilde C\big( \|E \|_{L^2(I;\cU)}^2
+ \|u-\cI_q u \|_{L^2(I;\cU)}^2 +  k^2\| (u-\cI_q u)'\|_{L^2(I,\cU)}^2\big), 
\]
with  a constant $\tilde C$ which depends only on $q$ and $\Gamma$. 
Using the triangle inequality and Theorem \ref{thmdiscr1} completes the proof.
\end{proof}

\ \\
\begin{lemma}
Let $k_n=k$ for all $n=1,\dots,N$. 
Assume that the solution $u$ of \eqref{problem} has regularity $u \in H^1(I;\cU)$. Let  $U\in\PP^b(I;\cU)$ be the solution of \eqref{Divp}-\eqref{Divp1}. The following holds: \begin{align*}
  & k^{-\frac12}\Big(\| (\cI_q u - U)^0_+ \|_{ \cH}^2 +\sum_{n=1}^{N-1} \| [\cI_q u-U]^{n} \|_{ \cH}^2 \Big)^\frac12   \\ & \leq    c( \|(u- \cI_q u)' \|_{L^2(I;\cU)} +  k^{-\frac{1}{2}}  \|\cI_q u- u \|_{L^2(I;\cU)})
\end{align*}
for some $c>0$ which depends only on $q,\gamma$ and $\Gamma$.
\end{lemma}
\begin{proof}
Let $E:=\cI_q u - U \in \PP^b(I; \cV)$. Let $E_n = \chi_{I_n }E$ for all $n=1,\dots,N$. Take for $2\leq n\leq N$, $\tilde E_{n} \in \PP_{q-1}(I_n; \cU)$ such that \[
\tilde E_{n}(t_{n-1} +t ) = E_{n-1}(t_{n-1} - t )\quad\text{for }t\in (0,k).
\]
Take $\tilde E_1=0$
and define $V_n := E_n-\tilde E_{n}$ for all $n$. 
 Note that $(V_1)_+^{0}=E^{0}_+$ and $(V_n)_+^{n-1}=[E]^{n-1}$ for $n=2,\dots,N$. 
For all $n=1,\dots,N$ and for $t\in I_n$ we have 
\begin{equation} \label{hh1}
\begin{split}
\| V_n(t) \|_{ \cH}  & \leq \| (V_n)_+^{n-1}\|_{ \cH} + \int_{t_{n-1}}^t \|V_n' \|_\cH \leq \| (V_n)_+^{n-1} \|_{ \cH} + \int_{I_n} \|V_n' \|_\cH\\
& \leq \|(V_n)_+^{n-1} \|_{ \cH} + \int_{I_n} \|E_n' \|_\cH + \int_{I_{n}} \|\tilde E_{n}' \|_\cH
\\& \leq \| (V_n)_+^{n-1} \|_{ \cH} + k^{\frac{1}{2}}\big( \|E' \|_{L^2(I_n;\cH)} +  \| E' \|_{L^2(I_{n-1};\cH)}\big),
\end{split}
\end{equation}
 with $I_0:=\emptyset$.
From the consistency property \eqref{consis1} and Lemma \ref{LemmaProj} we obtain the Galerkin relation 
\[
D_\cH(E, X ) + \int_I a(u-U, X )  = 0 \quad \text{for all}~~X \in \PP^b(I;\cV).
\]
If we take $X=\sum_{n=1}^N V_n$, we obtain 
\begin{equation} \label{hh2}
\| E^{0}_+ \|_{ \cH}^2 + \sum_{n=1}^{N-1} \|[E]^n\|^2_\cH =- \int_{I} (E', X )_\cH - \int_{I} a(u-U, X ).
\end{equation}
Using \eqref{hh1} we obtain for the first term on the right-hand side:
\begin{align*}
&\left|\int_{I} (E', X )_\cH\right| \leq \sum_{n=1}^N \int_{I_n} |(E', V_n )_\cH |  \leq  \sum_{n=1}^N \|E'\|_{L^1(
I_n;\cH)}\|V_n\|_{L^\infty(I_n;\cH)} 
\\ & \leq k^{\frac{1}{2}} \sum_{n=1}^N  \|E'\|_{L^2 (
I_n;\cH)} \big(\|(V_n)_+^{n-1}\|_{ \cH} + k^{\frac{1}{2}} (\|E' \|_{L^2(I_n;\cH)} + \| E' \|_{L^2(I_{n-1};\cH)}) \big)
\\ & \leq  \frac{k}{2} \|E'\|_{L^2(I;\cH)}^2 
+\frac{1}{2}\| E^0_+ \|_{ \cH}^2 + \frac{1}{2}\sum_{n=1}^{N-1} \|[E]^n\|^2_\cH
+ k \|E' \|_{L^2(I;\cH)}^2.
\end{align*}
The second term can be bounded as follows:
\begin{align*}
\left|\int_{I} a(u-U, X ) \right| & \leq \sum_{n=1}^N \Gamma \|u-U\|_{L^2(
I_n;\cU)}\|V_n\|_{L^2(I_n;\cU)} \\ & \leq \sum_{n=1}^N
\Gamma \|u-U\|_{L^2(
I_n;\cU)}(\|E_n \|_{L^2(I_n;\cU)} + \|\tilde E_n \|_{L^2(I_{n};\cU)})
\\ & \leq 
\frac \Gamma 2 \|u-U\|_{L^2(
I;\cU)}^2 + \Gamma \|E \|_{L^2(I;\cU)}^2.
\end{align*}
Using these bounds in \eqref{hh2} we get
\begin{align*}
  &  k^{-1} \big( \| (\cI_q u)^0_+ - U^0_+ \|_{ \cH}^2 + \sum_{n=1}^{N-1} \|[E]^n\|^2_\cH)  \\ & \leq C\big( \|E' \|_{L^2(I;\cH)}^2  + k^{-1}(\|E \|_{L^2(I;\cU)}^2+  \|u-U\|_{L^2(
I;\cU)}^2) \big)
\end{align*}
 with a constant $C$ which depends only on $\Gamma$.
Using a triangle inequality and the results in Lemma~\ref{lem89} and Theorem \ref{thmdiscr1}, we obtain the desired result.
\end{proof}

\begin{theorem}
Let $k_n=k$ for all $n=1,\dots,N$. 
Assume that the solution $(u,p)$ of \eqref{problemmixed} has regularity $u \in H^1(I;\cU)$. Let  $(U,P)\in\PP^b(I;\cU)$ be the solution of \eqref{Divpmixed}-\eqref{Divp1mixed}.
Let $\cI_q^{L^2}$ be the $L^2$-orthogonal projection from $L^2(I)$ to $\PP^b(I)$. 
The following holds: 
\begin{align*}
 \| p-P\|_{ L^2(I;\cQ_1)}  \leq &
 c \big(k^{\frac{1}{2}}\|(\cI_q u- u)' \|_{L^2(I;\cU)}+ k^{-\frac{1}{2}}\|\cI_q u- u \|_{L^2(I;\cU)} \big) 
\\ & +  \|p-\cI^{L^2}_q p \|_{L^2(I;\cQ_1)}
\end{align*}
with a constant $c$ which depends only on $q, \gamma, \Gamma$, $\beta$ from \eqref{assumpB1AA} and $c_\cH := \sup_{x\in \cU}\frac{\|x\|_\cH}{\|x\|_\cU}$. 

\end{theorem}
\begin{proof}
From the consistency property \eqref{consis1mixed} we obtain the Galerkin relation 
\[
D_\cH(u-U, X ) + \int_I a(u-U, X )  + \int_I  b_1(X,p-P) = 0 \quad \text{for all}~~X \in \PP^b(I;\cV_2).
\]
Let $\tilde P = \cI_q^{L^2} p$ be the orthogonal projection of $p\in L^2(I;\cQ_1)$ to $\PP^b(I,\cQ_1)$. From the inf-sup property \eqref{assumpB1AA} it follows that for all $t \in I$ there
exists a unique $\hat V(t) \in \cV^{\perp_{\cV_2}} \subset \cV_2$ such that:
\begin{align*}
 b_1(\hat V(t),q)& =(\tilde P(t)-P(t),q)_{\cQ_1} \quad \text{for all}~~q \in \cQ_1 \\
\text{\rm and}~~ \|\hat V(t)\|_{\cU} &\leq \frac{1}{\beta}\|\tilde P(t)-P(t)\|_{\cQ_1}.
\end{align*}
This, $ \tilde P - P \in \PP^b(I;\cQ_1)$ and a tensor-product argument imply $\hat V \in \PP^b(I;\cV_2)$.
Let $E=\cI_q u-U$. By Lemma \ref{LemmaProj} we have
\begin{equation}
D_\cH(E, \hat V ) + \int_I  a(u-U, \hat V )+  \int_I  b_1( \hat V,  p -P)   = 0.
\label{testWdeltaP-3}
\end{equation}
We first consider the first term of \eqref{testWdeltaP-3}. Using the continuous embedding $\cU \hookrightarrow \cH$ we get
 \[
 \left| \int_I (E', \hat V )_\cH\right| \leq  \|E'\|_{L^2(I;\cH)} \|\hat V \|_{L^2(I;\cH)} \leq c_\cH  \|E'\|_{L^2(I;\cH)} \|\hat V \|_{L^2(I;\cU)}.
\]
and using the inverse estimate $\|\hat V^{n}_+  \|_\cU \leq  Ck^{-\frac{1}{2}}  \|\hat V\|_{L^2(I_{n+1};\cU)}$, with  a constant $C$ which depends only on $q$, we get
\begin{align*}
 & |(E^0_+,V_+^0)_\cH|+\sum_{n=1}^{N-1} |([E]^{n},\hat V_+^{n})_\cH |  \leq  c_\cH \|E^{0}_+ \|_\cH\|\hat V^{0}_+\|_\cU + c_\cH \sum_{n=1}^{N-1} \|[E]^n\|_{\cH}\| \hat V^{n}_+\|_\cU
\\ &\leq   C c_\cH k^{-\frac{1}{2}} \|E^{0}_+ \|_\cH \|\hat V\|_{L^2(I_1;\cU)} 
 + Cc_\cH   k^{-\frac{1}{2}} \sum_{n=1}^{N-1}  \|[E]^{n}\|_{\cH}  \|\hat V\|_{L^2(I_{n+1};\cU )}
\\ & \leq  Cc_\cH  k^{-\frac{1}{2}} \big(\|E^{0}_+ \|_\cH^2+ \sum_{n=1}^{N-1} \|[E]^{n}\|_{\cH}^2 \big)^\frac12
\|\hat V \|_{L^2(I;\cU)}.
\end{align*}
For the second term, we have 
\[
\left|\int_I a(u-U, \hat V ) \right| \leq \Gamma \|u-U \|_{L^2(I;\cU)} \|\hat V \|_{L^2(I;\cU)}. 
\]
For the third term we have 
\begin{align}
\|\tilde P-P\|_{L^2(I;\cQ_1)}^2 & = \int_I (\tilde P-P, \tilde P- P)_{\cQ_1} = \int_I (\tilde P-P, p- P)_{\cQ_1} \nonumber 
 = \int_I b_1 (\hat V, p- P)   \nonumber .
\end{align}
Combing these results with the definition of $\hat V$, we get 
\begin{align*}
 & \|\tilde P-P\|_{L^2(I;\cQ_1)}^2  =  \int_I b_1 (\hat V, p- P)= -D_\cH(E, \hat V ) - \int_I  a(u-U, \hat V ) \\
 & \leq c \Big( \|E'\|_{L^2(I;\cH)} +  k^{-\frac{1}{2}}  \big(\|E^{0}_+ \|_\cH^2+ \sum_{n=1}^{N-1} \|[E]^{n}\|_{\cH}^2 \big)^\frac12 + \|u-U \|_{L^2(I;\cU)}\Big) \|\hat V \|_{L^2(I;\cU)} \\
 & \leq \frac{c}{\beta} \Big(\|E'\|_{L^2(I;\cH)}+ k^{-\frac{1}{2}}  \big(\|E^{0}_+ \|_\cH^2+ \sum_{n=1}^{N-1} \|[E]^{n}\|_{\cH}^2 \big)^\frac12 + \|u-U \|_{L^2(I;\cU)}\Big)\|\tilde P-P\|_{L^2(I;\cQ_1)}.
\end{align*}
Combining this with the previous two lemmas, Theorem \ref{thmdiscr1} and the triangle inequality completes the proof.
\end{proof}

\ \\[1ex]
Results for the projection error $\|p -\cI_q^{L^2} p\|_{L^2(I;\cQ_1)}$ with respect to the orthogonal projection are standard. Results for the projection errors $\|(\cI_q u- u)' \|_{L^2(I;\cU)}$ and $\|\cI_q u- u \|_{L^2(I;\cU)}$ are also standard, cf.~Subsection \ref{sectglobal}. Using these we obtain the following sub-optimal discretization error bound.
\begin{theorem}\label{Thmsuboptimal}
Let $(u,p)$ be the solution of \eqref{problemmixed} and $(U,P)$  the solution of \eqref{Divpmixed}-\eqref{Divp1mixed}.
Assume that the solution $(u,p)$  has smoothness properties $u \in  H^m(I;\cU)$, $p \in H^m(I;\cQ_1)$,  for an $m$ with $1 \leq m \leq q$. The following holds:
\begin{equation} \label{CeaboundLMbad}  
\|p-P\|_{L^2(I;\cQ_1)} \leq c k^{m-\frac{1}{2}} \big(\|u^{(m)}\|_{L^2(I;\cU)} + k^\frac{1}{2} \|p^{(m)}\|_{L^2(I;\cQ_1)}\big) 
\end{equation}
with a constant $c$ which depends only on $q,\gamma,\Gamma,\beta$.
\end{theorem}

\subsection{Optimal result}
Below we again use the notation introduced in Subsection~\ref{sectmixed}, in particular the bilinear form $b_1(\cdot,\cdot)$ and the corresponding operator $B_1: \cU \to Q_1'$, with $\ker(B_1)=\cV_1$. 
We introduce the $\cH$-orthogonal projection $P_{\cH}: \cH \rightarrow \overline{\cV}^{\cH}=:\cV_\cH$ onto the closed subspace $\cV_\cH$ of $\cH$, and define $P_{\cH}^\perp:={\rm id_\cH}-P_{\cH}:\, \cH \to \cH$. 
The following (regularity type) assumption on  $P_\cH$ is crucial for our analysis.
\begin{assumption}\label{AssContP} $P_\cH \cV_2 \subset \cV_1$ and there exists a constant $c_P$ such that
\[
 \|P_{\cH}v\|_\cU \leq c_P\|v \|_\cU \quad \text{for all}~~v \in \cV_2.
\]
\end{assumption}
\emph{In the remainder of this subsection  we assume that Assumption~\ref{AssContP} is satisfied}. Using this assumption we derive an optimal discretization error bound for the Lagrange multiplier. 
\begin{example} \rm \label{H2regularity}
We consider this assumption for Example \ref{example3}. In this case we have $\cU=H^1(\Omega)^d$, $\cH=L^2(\Omega)^d$, $\cV_1=\{\, u \in H^1(\Omega)^d~|~\Div u=0\,\}$, $\cV_2=H_0^1(\Omega)^d$ and $\cV=\cV_1 \cap \cV_2=\{\, u \in H_0^1(\Omega)^d~|~\Div u=0\,\}$. 
 The space $\cV_\cH $ is given by \[
\cV_\cH = \overline{\cV}^\cH = \{\,u\in L^2(\Omega)^d~|~ \Div u =0 , u\cdot n = 0\text{ on }\partial \Omega \,\},
\]
where $n$ denotes the normal of $\partial \Omega$, cf. \cite[p. 29]{GiraultRaviart}. 
The projection $P_\cH$ (also known as the Leray projector) can be characterized as follows. For given  $u\in L^2(\Omega)^d$, $P_\cH u = u +\nabla q$, with $q$ such that $(u+\nabla q,\nabla w)_{L^2(\Omega)^d}=0$ for all $w \in H^1(\Omega)/\Bbb{R}$. When $u \in H^1_0(\Omega)^d$, this $q$ solves the Poisson problem with homogeneous Neumann boundary conditions:  \begin{eqnarray*}
-\Delta q &=&\Div u \quad \text{on}~~\Omega \\
\frac{\partial q}{\partial n} &=& 0  \quad\text{on }~~\partial \Omega \\
 \int_\Omega q &=&0.
\end{eqnarray*}
We \emph{assume $H^2$-regularity of this problem}, hence, 
\[
 \|\nabla q\|_{H^1(\Omega)^d} \leq  \|q\|_{H^2(\Omega)}\leq c \|\Div u\|_{L^2(\Omega)} \leq  c\|u\|_{H^1(\Omega)^d},
\]
 which implies that $P_\cH: H^1_0(\Omega)^d \rightarrow H^1(\Omega)^d$ is continuous. This and $\div P_\cH u = \div u + \Delta q =0$ imply that Assumption \ref{AssContP} holds. 
{This $H^2$-regularity assumption, which holds when $\Omega$ is convex  or has a $C^2$-boundary, is exactly the same as in \cite[Assumption 3.1]{GuberovicSchwabStevenson2014}. There it is used in the well-posedness analysis of time-dependent Stokes equations, cf. \cite[Theorem 3.5]{GuberovicSchwabStevenson2014}. }
\end{example}
\ \\
\begin{theorem}\label{LagBound}
Assume that the solution $(u,p)$ of \eqref{problemmixed} has regularity $u \in H^1(I;\cH)$ and $Au\in L^2(I;\cV_\cH)'$. Let  $(U,P)$ be the solution of \eqref{Divpmixed}-\eqref{Divp1mixed}.
The following holds: \[
 \| p-P\|_{ L^2(I;\cQ_1)}  \leq   \frac{\Gamma(1+ c_P)}{\beta} \|u-U\|_{L^2(I;\cU)} +  \|p-\cI^{L^2}_q p \|_{L^2(I;\cQ_1)},
\]
where $\cI_q^{L^2}$ denotes the $L^2$-orthogonal projection from $L^2(I)$ to $\PP^b(I)$.
\end{theorem}
\begin{proof}
From the consistency property \eqref{consis1mixed} we obtain the Galerkin relation 
\[
D_\cH(u-U, X ) + \int_I a(u-U, X )   = 0 \quad \text{for all}~~X \in \PP^b(I;\cV).
\]
The regularity assumptions $u',Au\in L^2(I;\cV_\cH)'$ and the density of $\cV$ in $\cV_\cH$ imply
$K(U,\cdot)\in \PP^b(I;\cV_\cH)' $ and 
\[
D_\cH(u-U, X ) + \int_I a(u-U, X )   = 0 \quad \text{for all}~~X \in \PP^b(I;\cV_\cH).
\]
From $P_\cH \cV_2 \subset \cV_\cH$, $P_\cH\cV_2 \subset \ker(B_1)$ and the consistency property \eqref{consis1mixed} we obtain the Galerkin relation 
\begin{equation} \label{gg}
D_\cH(u-U, X ) + \int_I a(u-U, X )  + \int_I  b_1(X,p-P) = 0 \quad \forall~X \in \PP^b(I;\cV_2 + P_\cH\cV_2).
\end{equation}
Let $\tilde P = \cI_q^{L^2} p$ be the orthogonal projection of $p\in L^2(I;\cQ_1)$ to $\PP^b(I,\cQ_1)$. From the inf-sup property \eqref{assumpB1AA} it follows that for all $t \in I$ there
exists a unique $\hat V(t) \in \cV^{\perp_{\cV_2}} \subset \cV_2$ such that:
\begin{equation} \label{gg1}
 \begin{split}
 b_1(\hat V(t),q)& =(\tilde P(t)-P(t),q)_{\cQ_1} \quad \text{for all}~~q \in \cQ_1 \\
\text{\rm and}~~ \|\hat V(t)\|_{\cU} &\leq \frac{1}{\beta}\|\tilde P(t)-P(t)\|_{\cQ_1}.
\end{split}
\end{equation}
This, $ \tilde P - P \in \PP^b(I;\cQ_1)$ and a tensor-product argument imply $\hat V \in \PP^b(I;\cV_2)$. From $P_\cH\cV_2 \subset \ker(B_1)$ it follows that for $v \in \cV_2$ we have
\[  b_1(v,p)=b_1(P_\cH^\perp v,p) \quad \text{for all}~~p \in \cQ_1. \]
Using this and \eqref{gg}, \eqref{gg1} we get
\begin{align}
\|\tilde P-P\|_{L^2(I;\cQ_1)}^2 & = \int_I (\tilde P-P, \tilde P- P)_{\cQ_1} = \int_I (\tilde P(t)-P(t), p(t)- P(t))_{\cQ_1}\,dt \nonumber \\
 & = \int_I b_1 (\hat V(t), p(t)- P(t))\,dt =  \int_I b_1 (P_\cH^\perp \hat V(t), p(t)- P(t))\,dt  \nonumber \\
 & =\int_I a(U-u, P_\cH^\perp \hat V) + D_\cH(U-u, P_\cH^\perp \hat V). \label{duality}
\end{align}
For the first term we get
\begin{align*}
 \int_I a(U-u, P_\cH^\perp \hat V) & \leq \Gamma \|U-u\|_{L^2(I;\cU)} (1+c_P) \|\hat V\|_{L^2(I;\cU)} \\ & \leq \frac{\Gamma(1+c_P)}{\beta}\|U-u\|_{L^2(I;\cU)} \|\tilde P-P\|_{L^2(I;\cQ_1)}.
\end{align*}
For the second term we use Lemma \ref{LemmaProj} and the  fact that $D_\cH(\cdot,\cdot)$ consists of sums and integrals of $\cH$ scalar products, hence 
\begin{equation} \label{oo} 
D_\cH(U-u, P_\cH^\perp \hat V)=D_\cH(U-\cI_q u, P_\cH^\perp \hat V)=  D_\cH(P_\cH^\perp(U-\cI_q u),\hat V) = 0,
\end{equation}
where the last equality follows from $U(t)-\cI_q u(t) \in \ker(B)=\cV$ and $P_\cH^\perp(U'(t)-u'(t))= \big(P_\cH^\perp(U-\cI_q u)\big)'(t)$. Thus we obtain
\[
 \|\tilde P-P\|_{L^2(I;\cQ_1)} \leq  \frac{\Gamma(1+c_P)}{\beta}\|U-u\|_{L^2(I;\cU)}.
\]
This, in combination with the triangle inequality concludes the proof.
\end{proof}
 \ \\[1ex]

Results for the projection error $\|p -\cI_q^{L^2} p\|_{L^2(I;\cQ_1)}$ with respect to the orthogonal projection are standard. Using these  and \eqref{CeaboundA} we obtain the following optimal discretization error bound.
\begin{theorem} \label{Thmpoptimal}
Let $(u,p)$ be the solution of \eqref{problemmixed} and $(U,P)$  the solution of \eqref{Divpmixed}-\eqref{Divp1mixed}.
Assume that the solution $(u,p)$  has smoothness properties $Au\in L^2(I;\cV_\cH)'$, $u \in  H^m(I;\cU)$, $p \in H^m(I;\cQ_1)$,  for an $m$ with $1 \leq m \leq q$. The following holds:
\begin{align}
\|p-P\|_{L^2(I;\cQ_1)} &\leq c \left(\sum_{n=1}^Nk_n^{2m} \left(\left(\frac{1+c_P}{\beta} \right)^2\|u^{(m)}\|_{L^2(I_n;\cU)}^2 + \|p^{(m)}\|_{L^2(I_n;\cQ_1)}^2 \right)\right)^{\frac{1}{2}}\nonumber
\\& \leq c k^m \left(\frac{1+c_P}{\beta} \|u^{(m)}\|_{L^2(I;\cU)} + \|p^{(m)}\|_{L^2(I;\cQ_1)}\right) \label{CeaboundLM}
\end{align}
for some $c>0$, which depends only on $q,\gamma$ and $\Gamma$.
\end{theorem}
\ \\
{
\begin{remark}\label{RemH2Reg} \rm
For the derivation of an optimal error bound for the Lagrange multiplier, as in Theorem~\ref{Thmpoptimal}, for the general abstract (saddle point) problem \eqref{problemmixed} Assumption~\ref{AssContP} is sufficient. For the specific case of the Stokes equations (Example \ref{example3}) such optimal error bounds can be derived under weaker assumptions than the $H^2$-regularity discussed in  Example~\ref{H2regularity}. The analysis, which is a topic of current research, uses additional specific properties of the Stokes problem. 
\end{remark}
}
\section{Fully discrete problem}\label{SectionFullDisc}
In applications of the abstract setting one typically has infinite dimensional spaces $\cU$, $\cV$, $\cQ$, $\cQ_1$, $\cV_2$ in the time-discrete problems \eqref{Divp}-\eqref{Divp1} (constrained formulation) and \eqref{Divpmixed}-\eqref{Divp1mixed} (mixed formulation). These  spaces are then replaced by finite dimensional (e.g. finite element) spaces to obtain a feasible method. Below we describe this fully discrete setting.
\\
\subsection{Fully discrete constrained formulation} \label{sectfullydiscr}
In this subsection we introduce the spacial discretization for the constrained problem \eqref{Divp}-\eqref{Divp1}. We assume that a conforming finite element method is used, i.e.~we have a finite dimensional subspace $\cU^h \subset \cU$. We define $\cV^h = \cU^h \cap \cV$ and $\cQ^h = B\cU^h \subset \cQ'$. 
We assume that we have \emph{spatial} projection operators $\bI_{\cQ^h}: \cQ' \rightarrow \cQ^h$ and  $\bI_{\cU^h}: \cU \rightarrow \cU^h$, which satisfy $\bI_{\cQ^h} B = B \bI_{\cU^h}$. If we assume some additional spatial regularity on $g$ and $u_0$, then it is sufficient if $\bI_{\cU^h}$ and $\bI_{\cQ^h}$ are only defined on subsets of $\cU$ and $\cQ'$ and $u_0\in \cD(\bI_{\cU^h})\subset \cU$, $g(t) \in \cD(\bI_{\cQ^h})\subset \cQ'$ for almost all $t\in I$. In applications,  these spatial projections typically are nodal interpolations, cf. Example~\ref{Extyp} below.

The fully  discrete version of \eqref{problem}  reads: find $U_h \in \PP^b(I;\cU^h)$ such that
\begin{align} 
K(U_h, X_h) & = \int_I f(X_h) + (\bI_{\cU^h} u_0,(X_h)_+^0)_\cH \quad \text{for all}~~X_h \in \PP^b(I;\cV^h) \label{FullDivp} \\
B U_h & = \cI_q \bI_{\cQ^h} g.  \label{FullDivp1}
\end{align}
The commutation property $\bI_{\cQ^h} B = B \bI_{\cU^h}$ implies the important consistency property $\cI_q B(U-\bI_{\cU^h} u)=0$ for the constraint discretization (in space and time).
The fully discrete problem \eqref{FullDivp}-\eqref{FullDivp1} gives rise to a system of linear equations. For practical purposes it is useful to have a (cheap) way to compute for $g_h \in \cQ^h$ a $G_h\in \cU^h$ with $B G_h = g_h$. We will denote this operator by $\bL_{\cQ^h}$. 
Provided one has such an operator, \eqref{FullDivp}-\eqref{FullDivp1} is equivalent to finding $W_h \in \PP^b(I;\cV^h)$ such that
\begin{equation}
K(W_h, X_h)  = \int_I f(X_h) + (\bI_{\cU^h} u_0,(X_h)_+^0)_\cH - K(\cI_{q} G_h, X_h) \quad \forall~X_h \in \PP^b(I;\cV^h) 
\label{FullDiscConst}
\end{equation}
with $G_h = \bL_{\cQ^h} \bI_{\cQ^h} g$. We then have $U_h=W_h + \cI_{q} G_h$. Note that in \eqref{FullDiscConst} both the trial and test space are $\PP^b(I;\cV^h)$.
The global in time equation \eqref{FullDiscConst} does not need to be solved in this form. It has a lower block triangular structure, cf.~definition of $K(\cdot,\cdot)$ and $D_\cH(\cdot,\cdot)$. Thus \eqref{FullDiscConst} can be solved sequentially by solving for each $n=1,\dots,N$
\begin{equation}
\begin{split}
&\int_{I_n} (W_{h,n}', X_h)_\cH + ( (W_{h,n})_+^{n-1},(X_h)_+^{n-1})_\cH + \int_{I_n} a(W_{h,n}, X_h)  
\\ &= \int_{I_n} f(X_h) + ( W_{h,n-1}^{n-1},(X_h)_+^{n-1})_\cH - \int_{I_n} ( G_h', X_h)_\cH - 
\int_{I_n} a(\cI_q G_h,X_h)
\end{split}
\label{nDiscConst}
\end{equation}
for all $X_h \in \PP_{q-1}(I_n;\cV^h)$, where $W_{h,n}=W_h|_{I_n}$ and $W_{h,0}^0 = \bI_{\cU^h} u_0 - G_h(0)$.
Typically $g$ is given explicitly, but $g'$ is not, thus $G_h'$ is also not known explicitly. Hence, it is often convenient to rewrite the term $\int_{I_n} ( G_h', X_h)_\cH$ using integration by parts.
Note that solving \eqref{nDiscConst} requires explicit computations in $\cV^h$ which may or may not be easy in a practical setting. 

\begin{example} \label{Extyp} \rm 
We outline the fully discrete problem for the specific case given in Example \ref{Exheat}.  Let $\Omega$ be a polygonal domain and, for a given simplicial triangulation $\TT_h$ and  $r\in \NN$,
\[
\cU^{h} = \{u \in H^1(\Omega)~|~u|_{T_s} \in  \PP_r(T_s) \text{ for all } ~T_s \in \TT_h\,\}
\]
a standard $H^1$-conforming finite element space.
In order to express \eqref{nDiscConst} as a linear system, we need an explicit basis of 
\[
\cV^{h} =\cU^{h} \cap H^1_0(\Omega= \{u \in H^1_0(\Omega)~|~ u|_{T_s} \in  \PP_r(T_s)\text{ for all } ~T_s \in \TT_h\, \}.
\]
Taking a nodal basis for $\cU^h$ and omitting the nodes on the boundary $\partial \Omega$ yields such a basis. The operator $\bL_{\cQ^h} \bI_{\cQ^h}|_{H^\frac12(\partial \Omega)\cap C(\partial \Omega)} $ can be taken to be a Lagrange interpolation on operator $\partial \Omega$, cf. \cite[Section 3.2.2]{ErnGuermond}. The operator $\bI_{\cU^h}|_{H^1(\Omega)\cap C(\Omega)}$, which is used for the initial data, can be taken to be a nodal  interpolation operator. For these choices the commutation property $\bI_{\cQ^h} B = B \bI_{\cU^h}$, with $B$ the trace operator on $\partial \Omega$, holds.
\end{example}
\ \\[1ex]
For the Stokes equation (Examples \ref{ExStokes}, \ref{example3}) one would need an explicit feasible  basis for the space 
\[
\cV^{h} = \{u \in H^1_0(\Omega)^d|\div u =0,~~u|_{T_s} \in  \PP_r(T_s)^d\text{ for all }~ T_s \in \TT_h\,\},
\]
i.e., of a space of divergence free finite elements. In many cases this is rather cumbersome.  
Furthermore, computing $\bL_{\cQ^h} \bI_{\cQ^h}$ would involve finding $G_h$ such that $\div G_h = \bI_{\cQ^h} g$, which  also has certain disadvantages.  These difficulties can be avoided by using the mixed method which will be presented in the next subsection.

\subsection{Fully discrete mixed formulation} \label{sectdiscrete}
In this subsection we introduce the spatial discretization for the mixed  problem \eqref{Divpmixed}-\eqref{Divp1mixed}.  As in the previous section,  we assume that a conforming finite element method is used, i.e. we have finite dimensional subspaces $\cU^h \subset \cU$ and $\cQ_1^h \subset \cQ_1$. We define $\cQ^h_2 = B_2 \cU^h \subset \cQ'_2$ and  $\cV^h_2 = \cU^h \cap \cV_2$. For existence of a unique solution we assume ($h$-dependent) inf-sup stability of the  pair $(\cV^h_2,\cQ_1^h)$ with respect to the constraint operator that corresponds to the implicitly treated constraints, i.e,  we assume that there is a $\tilde \beta_h > 0$  such that
\begin{equation} \label{assumpB1AAh}
\inf_{q_1 \in\cQ_1^h} \sup_{u_h \in \cV_2^h} \frac{b_1(u_h,q_1)}{\|u_h\|_{\cU} \|q_1\|_{\cQ_1}} \geq \tilde \beta_h.
\end{equation}
Due to the different ways of treating $b_1$ and $b_2$ we will use \eqref{split1A}-\eqref{split2A} instead of \eqref{Divp1mixed}. Note that $\cQ_1^h$ is taken to be a subspace of $\cQ_1$, which is the solution space for the Lagrange multiplier, whereas $\cQ_2^h$ is taken to be a subspace of $\cQ_2'$ (as in the previous subsection). We assume that we have a projection $\bI_{\cQ^h_2}: \cQ_2' \rightarrow \cQ_2^h$ and a projection $\bI_{\cU^h}: \cU \rightarrow \cU^h$, which satisfy $\bI_{\cQ_2^h} B = B \bI_{\cU^h}$. 

The fully  discrete version of \eqref{problem}  reads: find $U_h \in \PP^b(I;\cU^h)$, $P_h\in \PP^b(I;\cQ^h_1)$ such that
\begin{align} 
K(U_h, X_h)+\int_I b_1(X_h,P_h) & = \int_I f(X_h) + (\bI_{\cU^h} u_0,(X_h)_+^0)_\cH \quad \text{for all}~~X_h \in \PP^b(I;\cV^h_2) \nonumber \\
\int_I b_1(U_h,R_h) & = \int_I (\cI_q g_1)(R_h) \label{MixedFullDivp2}
\quad \text{for all}~~R_h \in \PP^b(I;\cQ^h_1) 
 \\
B_2 U_h  & = \cI_q \bI_{\cQ_2^h} g_2  \nonumber
\end{align}
 The fully discrete problem \eqref{MixedFullDivp2} gives rise to a system of linear equations. The explicit constraints are treated in the same way as in Subsection~\ref{sectfullydiscr}, i.e., we assume a (cheap) way to compute for $g_h \in \cQ^h_2$ a $G_h\in \cU^h$ with $B_2 G_h = g_h$. We denote this operator by $\bL_{\cQ^h_2}$. 
Provided one has such an operator, \eqref{MixedFullDivp2} is equivalent to finding $W_h \in \PP^b(I;\cV^h_2)$, $P_h\in \PP^b(I;\cQ^h_1)$ such that for all $X_h \in \PP^b(I;\cV^h_2)$:
\begin{align}
& K(W_h, X_h) +  \int_I b_1(X_h, P_h) = \int_I f(X_h) + (\bI_{\cU^h} u_0,(X_h)_+^0)_\cH - K(\cI_{q} G_h, X_h) 
\label{MixedFullDiscConst}
\\  & \int_I b_1(W_h,R_h)  = \int_I (\cI_q g_1)(R_h) - \int_I b_1(\cI_q G_h,R_h)
\quad
\text{for all}~~R_h \in \PP^b(I;\cQ^h_1), \label{MixedFullDiscConst2}
\end{align}
with $G_h = \bL_{\cQ^h_2} \bI_{\cQ^h_2} g_2$. Assumption \eqref{assumpB1AAh} ensures that this system has a unique solution. We have $U_h=W_h + \cI_{q} G_h$.
The global in time equation \eqref{MixedFullDiscConst}-\eqref{MixedFullDiscConst2}
can be solved sequentially by solving for each $n=1,\dots,N$
\begin{eqnarray}
&&\begin{split}
&\int_{I_n} (W_{h,n}', X_h)_\cH + ( (W_{h,n})_+^{n-1},(X_h)_+^{n-1})_\cH + \int_{I_n} a(W_{h,n}, X_h)   + \int_{I_n} b_1(X_h, P_{h,n})
\\ &= \int_{I_n} f(X_h) + ( W_{h,n-1}^{n-1},(X_h)_+^{n-1})_\cH - \int_{I_n} ( G_h', X_h)_\cH - 
\int_{I_n} a(\cI_q G_h,X_h)
\end{split}
\label{nMixedDiscConst}
\\ &&
\int_{I_n} b_1(W_{h,n},R_h)  = \int_{I_n} (\cI_q g_1)(R_h) - \int_{I_n} b_1(\cI_q G_h,R_h)
\label{nMixedDiscConst2}
\end{eqnarray}
for all $X_h \in \PP_{q-1}(I_n;\cV^h_2), R_h \in \PP_{q-1}(I_n;\cQ^h_1) $, where $W_{h,n}=W_h|_{I_n}$, $P_{h,n}=P_h|_{I_n}$ and $W_{h,0}^0 = \bI_{\cU^h} u_0 - G_h(0)$.
Solving the saddle point problem \eqref{nMixedDiscConst}-\eqref{nMixedDiscConst2} requires explicit computations in $\cV^h_2$ and $\cQ_1^h$ but not in $\cV^h$. For the term involving $G_h'$ integration by parts is convenient, cf. the discussion in the previous subsection. 

In the specific case of the Stokes problem (on a polygonal domain $\Omega$) as discussed in Example \ref{example3}, an obvious choice is the   Hood-Taylor pair:
\begin{eqnarray*}
\cQ_1^{h} &=& \{p \in H^1(\Omega)\cap L^2_0(\Omega) ~|~ p|_{T_s} \in \PP_{r-1}(T_s) \text{ for all } ~T_s \in \TT_h\, \},
\\
\cU^{h} &=& \{u \in H^1(\Omega)^d~|~ u|_{T_s} \in  \PP_r(T_s)^d \text{ for all } T_s \in \TT_h\,\},
\end{eqnarray*}
where $2\leq r\in \NN$.
In order to express \eqref{nMixedDiscConst}-\eqref{nMixedDiscConst2} as a linear system, we need an explicit basis of $\cQ_1^{h}$ and of
\[
\cV^{h}_2 = \cU^h \cap \cV_2=\{u \in H^1_0(\Omega)^d~|~u|_{T_s} \in  \PP_r(T_s)^d \text{ for all } ~ T_s \in \TT_h\, \}.
\]
In the same way as in the previous subsection, a basis for $\cV^{h}_2$ can be obtained by taking a nodal basis for $\cU^h$ and omitting the nodes on the boundary $\partial \Omega$. The operators $\bL_{\cQ^h_2} \bI_{\cQ^h_2} $ and $\bI_{\cU^h}$ can also be taken as in the previous subsection. For these choices of finite element spaces the stability Assumption \eqref{assumpB1AAh} is satisfied, even with $\tilde \beta_h$ independent of $h$.

\section{Numerical experiments}\label{SectionNum}
 We consider two Stokes problems with a known analytic solution to validate the results from the error analysis numerically and to illustrate certain phenomena. 
The first example is a  Stokes problem with a non-zero divergence condition and  a non-zero Dirichlet boundary condition. Despite the fact that the results of the error analysis are derived (only) in a semi-discrete setting, we see the predicted convergence rates (w.r.t. time discretization) also in the fully discrete setting. We use this Stokes example not only to validate our theoretical results but also to show that one does \emph{not} obtain optimal results if  the projection $\cI_q$ in \eqref{MixedFullDivp2} is omitted, cf.~Remark \ref{RemarkSubOptIq}.
The second example is a  Stokes  interface problem with a stationary interface. The density and diffusion coefficients are piecewise constant and discontinuous across the interface. We take a homogeneous divergence condition and a non-zero Dirichlet boundary condition (in/out-flow). This less smooth example is also used to illustrate a certain effect related to the semi-norms $|\cdot|_\ell$, which were introduced in Section \ref{SubSectionOptSuper}.

The methods are implemented in the software package DROPS, cf.~\cite{DROPS}.

\subsection{Stokes Problem}\label{NumOne} 
We {consider a problem as in Example~\ref{example3}. 
We take $\Omega = (-1,1)^3$ and a time interval $(0,1)$}. Different from Example \ref{example3}, in the  diffusion part we use the symmetrized gradient $D u =\nabla u+ \nabla u ^T$, i.e., 
$a(u,v)=\int_\Omega Du : Dv$. This difference is not essential. We use the symmetrized gradient because in view of the interface problem considered below it is more natural.  The constraint operators are $B_1=\Div$ and $B_2={\rm tr}_{|\partial \Omega}$. 
We take \[
u = \left( \begin{array}{c}
(x^2+1)(z+y)\sin(4t)\\
(y^2+1)(z+x)\sin(4t)\\
(z^2+1)(x+y)\sin(4t)
\end{array}
\right),\quad p = e^t(x^2+y^2+z^2)
\]
and we discretize the problem \begin{eqnarray}
u' {-} \Div D u + \nabla p &=& f \label{Exp1-1}
\\
\Div u &=& g_1 \label{Exp1-2}
\\
u|_{\partial \Omega} &= &g_2 \label{Exp1-3}
\end{eqnarray}
for the appropriate right hand sides. Note that $g_1\neq 0$, $g_2\neq 0$ and both are time-dependent.
 We take a uniform step size $k=\frac{1}{N}$ in time and  $q=2$.
For the discretization in space we take a triangulation of $\Omega$. To obtain the triangulation $\TT_h$ the domain $\Omega$ is divided into cubes with side length $h:=\frac{1}{N_S}$ and each of the cubes is divided into six tetrahedra. We use the $\PP_2$-$\PP_1$ Hood-Taylor pair
\begin{eqnarray}
\cQ_1^{h} &=& \{p \in H^1(\Omega)~ |~p|_{T_s} \in \PP_1(T_s)\text{ for all }~ T_s \in \TT_h \,\},\label{defQ1}
\\
\cU^{h} &=& \{u \in H^1(\Omega)^3~|~ u|_{T_s} \in  \PP_2(T_s)^3\text{ for all }~ T_s \in \TT_h\,\}. \nonumber
\end{eqnarray}
 To assemble the linear system  {\eqref{nMixedDiscConst}-\eqref{nMixedDiscConst2}} we  use the method described in the previous section. We denote the standard nodal basis of $\cU^h$ by $\{ \psi_{1}, \dots \psi_{N}\} = \Psi$. From the we extract a nodal basis of $\cV_2^h$: $\Psi_0 = \{\psi~|~ \psi\in \Psi \cap H^1_0(\Omega)^3 \}$ and we take $\Psi_{\partial \Omega} = \Psi\setminus \Psi_0$. We treat the explicit constraint $u|_{\partial\Omega} = g_2$ by taking $G_h = \sum_{\psi \in \Psi_{\partial \Omega} }  c(\psi)\psi$ such that $G_h|_{\partial\Omega}(x)=g_2(x)$ for all nodal points $x\in\partial \Omega$. {To compute the  matrices and the right-hand sides in \eqref{MixedFullDiscConst}-\eqref{MixedFullDiscConst2}  we exploit the tensor-product structure of the mesh. A fifth order spatial quadrature rule is used and combined  with a four point Gauss quadrature in time.}
\begin{table}[ht!]
\caption{
\label{TableL2H1-good1}
Error in $L^2\otimes H^1$-norm between $u$ and the solution of \eqref{MixedFullDivp2}. The estimated temporal (spacial) order of convergence $EOC_T$ ($EOC_S$) is computed using the last row (column).
}
\begin{tabular}{r|lllllll}
$N_S$\textbackslash $N$ &4&8&16&32&64&128&$EOC_S$\\
\hline
4&0.47350&0.13324&0.06644&0.05980&0.05936&0.05934&\\
8&0.46973&0.11950&0.03076&0.01038&0.00744&0.00721&3.04024\\
16&0.46966&0.11929&0.02991&0.00753&0.00207&0.00100&2.84747\\
32&0.46966&0.11928&0.02990&0.00748&0.00187&0.00048&1.06226\\
$EOC_T$&&1.97726&1.99607&1.99831&1.99965&1.96332&
\end{tabular}

\caption{\label{TableL2H1-bad1}
Error in $L^2\otimes H^1$-norm between $u$ and the solution of \eqref{MixedFullDivp2}, if we omit the projection $\cI_q$.
The estimated temporal (spacial) order of convergence $EOC_T$ ($EOC_S$) is computed using the last row (column).
}

\begin{tabular}{r|lllllll}
$N_S$\textbackslash $N$ &4&8&16&32&64&128&$EOC_S$\\
\hline
4&0.82210&0.21769&0.07950&0.06081&0.05943&0.05934&\\
8&1.22252&0.31238&0.07942&0.02133&0.00883&0.00732&3.01999\\
16&1.78108&0.45328&0.11427&0.02887&0.00737&0.00206&1.82496\\
32&2.55929&0.64970&0.16318&0.04099&0.01032&0.00261&-0.33854\\
$EOC_T$&&1.97791&1.99326&1.99318&1.99008&1.98258&
\end{tabular}

\end{table}

The Stokes problem \eqref{Exp1-1} - \eqref{Exp1-3} is discretized by \eqref{MixedFullDiscConst}-\eqref{MixedFullDiscConst2}  as described in the previous section. The error $\|u-U_h\|_{L^2\otimes H^1}$ is given in Table~\ref{TableL2H1-good1}. In this table we see that the error is of order $\mathcal{O}(k^2+h^2)$ (cf. diagonals in the table) which is optimal in time, as predicted by Theorem \ref{mainthm1}, and in space. If we omit the projection $\cI_q$ in \eqref{MixedFullDivp2}, then we obtain a sub-optimal results of order $\cO(k^2h^{-\frac{1}{2}} + h^2)$ in Table~\ref{TableL2H1-bad1}.  Note that for a fixed temporal discretization refinement in space results in divergence of order $\mathcal{O}(h^{-1/2})$. This behavior can be observed rather clearly in the first two columns ($N=4,8$) of Table~\ref{TableL2H1-bad1}, where the time discretization error dominates.  The $h^{-\frac12}$ behavior can  be explained using some further analysis, which we do not present here (since it involves the analysis of the 
spatial discretization). 
We also see that along the diagonal we only have a space-time convergence order of 1.5.

We are also interested in superconvergence. Theorem \ref{thmSuper} can be applied not only on $[0,T]$, but also on any subinterval $[0,t_n]$. We thus are interested in the convergence order for the maximal nodal error $\max_{n=1,\dots, N} \|u^n-U_h^n\|_{L^2}$. 
We expect a $\mathcal{O}(k^3+h^3)$ convergence order: optimal in time, due to Theorem \ref{thmSuper} and optimal in space. This is consistent with the numerical results in Table \ref{TableLinfL2-good1}, however, the spatial error dominates.  If we omit the projection operator $\cI_q$, we again obtain results that are  sub-optimal (Table \ref{TableLinfL2-bad1}). In particular we do  not observe superconvergence. Related to this,  the nodal (in time) discretization error can be several orders of magnitude smaller due to the use of the projection operator $\cI_q$, e.g., the results for $N=4, N_S=32$ in the Tables~\ref{TableLinfL2-good1}, \ref{TableLinfL2-bad1}.

\begin{table}[ht!]
{
\caption{\label{TableLinfL2-good1}
Maximal nodal error $\|u^n-U_h^n\|_{L^2}$, where $U_h$ is the solution of  \eqref{MixedFullDivp2}.
The estimated temporal (spacial) order of convergence $EOC_T$ ($EOC_S$) is computed using the last row (column).
}
\begin{tabular}{r|lllllll}
$N_S$\textbackslash $N$ &4&8&16&32&64&128&$EOC_S$\\
\hline
4&0.02137&0.02343&0.02343&0.02346&0.02349&0.02349&\\
8&0.00252&0.00276&0.00277&0.00277&0.00277&0.00277&3.0834\\
16&3.09e-04&3.39e-04&3.40e-04&3.40e-04&3.40e-04&3.40e-04&3.0265\\
32&3.94e-05&4.34e-05&4.38e-05&4.43e-05&4.23e-05&4.23e-05&3.0077\\
$EOC_T$&&-0.1384&-0.0154&-0.0150&0.0666&-5.8e-05&
\end{tabular}
}
\caption{\label{TableLinfL2-bad1}
Maximal nodal error $\|u^n-U_h^n\|_{L^2}$, where $U_h$ is the solution of  \eqref{MixedFullDivp2}, if we omit the projection $\cI_q$.
The estimated temporal (spacial) order of convergence $EOC_T$ ($EOC_S$) is computed using the last row (column).
}
\begin{tabular}{r|lllllll}
$N_S$\textbackslash $N$ &4&8&16&32&64&128&$EOC_S$\\
\hline
4&0.28772&0.06873&0.02488&0.02247&0.02317&0.02341&\\
8&0.37160&0.09146&0.02327&0.00611&0.00293&0.00274&3.0948\\
16&0.41012&0.10133&0.02600&0.00651&0.00164&5.10e-03&2.4255\\
32&0.42844&0.10590&0.02719&0.00683&0.00170&4.27e-04&0.2570\\
$EOC_T$&&2.0164&1.9615&1.9932&2.0028&1.9974
\end{tabular}

\end{table}

We are also interested in the  $L^2\otimes L^2$ error for the pressure $p$. 
Recall that we have two results. Theorem \ref{CeaboundLMbad} provides a general, but sub-optimal result. Theorem \ref{CeaboundLM} provides an optimal result, if Assumption \ref{AssContP} holds. In the setting of this experiment this assumption  is satisfied, cf. Example \ref{H2regularity}. Due to the convexity of $\Omega$ we have $H^2$-regularity for the Poisson problem. We observe an optimal convergence order $\cO(k^2+h^2)$ in Table \ref{TablePL2L2-good1}. Again, if we do not use the projection operator $\cI_q$, we lose the optimal convergence order, see Table \ref{TablePL2L2-bad1}. 
\begin{table}
\caption{\label{TablePL2L2-good1}
Error in $L^2\otimes L^2$-norm between $p$ and the solution of  \eqref{MixedFullDivp2}.
The estimated temporal (spacial) order of convergence $EOC_T$ ($EOC_S$) is computed using the last row (column).
}
\begin{tabular}{r|lllllll}
$N_S$\textbackslash $N$ &4&8&16&32&64&128&$EOC_S$\\
\hline
4&0.42505&0.30505&0.29582&0.29525&0.29522&0.29522&\\
8&0.31218&0.09285&0.05403&0.05064&0.05040&0.05039&2.5507\\
16&0.30845&0.07883&0.02237&0.01204&0.01089&0.01083&2.2185\\
32&0.30828&0.07813&0.01984&0.00589&0.00286&0.00261&2.0540\\
$EOC_T$&&1.9804&1.9771&1.7534&1.0400&0.1347&
\end{tabular}
\caption{\label{TablePL2L2-bad1}
Error in $L^2\otimes L^2$-norm between $p$ and the solution of  \eqref{MixedFullDivp2}, if we omit the projection $\cI_q$.
The estimated temporal (spacial) order of convergence $EOC_T$ ($EOC_S$) is computed using the last row (column).
}
\begin{tabular}{r|lllllll}
$N_S$\textbackslash $N$ &4&8&16&32&64&128&$EOC_S$\\
\hline
4&1.20438&0.69114&0.43399&0.33535&0.30565&0.29784&\\
8&1.32184&0.71175&0.36637&0.18878&0.10372&0.06762&2.1390\\
16&1.35591&0.73062&0.37476&0.18879&0.09479&0.04816&0.4896\\
32&1.36231&0.73455&0.37695&0.19011&0.09513&0.04755&0.0184\\
$EOC_T$&&0.8911&0.9625&0.9875&0.9989&1.0004&
\end{tabular}

\end{table}

\subsection{A Stokes interface problem}\label{SubSectionStokesInterface}
We consider the same setting as in the previous experiment. {However,  we now use discontinuous piecewise constant density and diffusion coefficients $\rho$ and $\mu$ and a spatial domain with a re-entrant corner $\Omega = (-1,1)^3\setminus [-\frac{1}{2},1)\times [0,1)^2$}.  In the previous experiment we have $\rho=\mu=1$ in $\Omega$. 
A stationary interface  separates two subdomains (phases) $\Omega_+ = \{(x,y,z)\in \Omega~|~ x\geq 1/7\}$ and $\Omega_- = \Omega \setminus \Omega_+$. The density and diffusion coefficients are $\rho = \mu = 5\chi_{\Omega_-} + \chi_{\Omega_+}$. 
The corresponding equations read
\begin{eqnarray}
\rho u' {-} \Div (\mu D u) + \nabla p &=& f  \quad \text{in}~~L^2(I;H^{-1}(\Omega)) \label{Exp2-1}
\\
\Div u &=& 0 \label{Exp2-2}
\\
u|_{\partial \Omega} &= &g \label{Exp2-3}.
\end{eqnarray} 
{ In this problem we have $a(u,v)=\int_\Omega \mu Du : Dv$ and the constraint operators $B_1$, $B_2$ are defined as in the previous experiment. 
In order to apply the theory from the previous sections, we replace the standard $L^2$ scalar product by $(u,v)_\cH=\int_\Omega \rho u\cdot v$, see Remark~\ref{remgenral}.}
Note that the jump in $\mu$ strongly influences the behavior of the semi-norms $|\cdot|_\ell$ which were introduced in Subsection \ref{SubSectionOptSuper}.
We take
\begin{eqnarray*}
p & = &\left(\frac{32}{7} t^2 z\right)\chi_{\Omega_+}(x,y,z),
\\u&= &\left( \begin{array}{l}
2t^2z(x^2+y^2)
\\
-4t^2xyz
\\
-2/3 t^2 x(x^2+3y^2)
\end{array}\right)
\end{eqnarray*}
and the corresponding $f$ and $g$.
{We use a similar tetrahedral triangulation of the domain and  obtain the space $\cU^h$  in the same way.} To account for the discontinuity in $p$ we use the following XFEM space, cf. \cite{Reusken}:
\[
  \cQ_1^{h,X}:= \mathcal{R}_+  \cQ_1^{h} \oplus  \mathcal{R}_-  \cQ_1^{h},
\]
where $\cQ_1^{h}$ is the standard finite element space of piecewise linears, cf.~\eqref{defQ1}, and $ \mathcal{R}_\pm:v\mapsto v \chi_{\Omega_\pm}$ is the restriction operator to the subdomain $\Omega_\pm$.
{ The basis for $\cV^h_2$ and the interpolation operator for the boundary data $g$ are  as in the previous experiment.}
For our particular choice of $u, p$ and $\mu$, we have that $f\in C^\infty(I;L^2(\Omega)^d)$. From this, $\rho u'\in C^\infty(I;L^2(\Omega))$ and $p\in C^\infty(I;L^2(\Omega))$ we can conclude that $\Div (\mu D u) \in C^\infty(I;\cV_\cH')$, where $\cV_\cH$ is defined as in Example \ref{H2regularity}. {\color{red}It follows that  $\int_0^T |\Div (\mu D u^{(2)})|_0^2 < \infty$.} We can expect at least a temporal  superconvergence rate of $2.5$ in Theorem \ref{mainthm2}. This is what we indeed  observe in Table \ref{TableLinfL2-good2}, namely a convergence of (approximately) order $\cO(k^{2.5}+h^3)$, where after a few temporal refinements the spatial error dominates. In the global $L^2\otimes H^1$-norm we observe an optimal temporal convergence order in Table \ref{TableL2H1-good2}.

Concerning the pressure discretization we note the following. For given $t\in I$ the pressure $p(t)$ can be expressed \emph{exactly} in the XFEM space. Hence,  the spatial convergence order of the pressure discretization is not relevant. { Due to Theorem \ref{Thmsuboptimal} we expect at least a temporal convergence order $\cO(k^{1.5})$. The actual temporal convergence order, which we see in Table \ref{TablePL2L2-good2}, is $\cO(k^2)$. This is despite the fact that the projection $P_\cH$ form Assumption \ref{AssContP} involves the discontinuous coefficient $\rho$ and the domain $\Omega$ has a re-entrant corner. The latter implies that  Assumption \ref{AssContP} does \emph{not} hold. For an explanation of this  we refer to Remark~\ref{RemH2Reg}. }

\begin{table}[ht!]
{
\caption{\label{TableL2H1-good2}
Error in $L^2\otimes H^1$-norm between $u$ and the solution of  \eqref{MixedFullDivp2}.
The estimated temporal (spacial) order of convergence $EOC_T$ ($EOC_S$) is computed using the last row (column).
}

\begin{tabular}{r|lllllll}
$N_S$\textbackslash $N$ &4&8&16&32&64&128&$EOC_S$\\
\hline
4&0.12001&0.09779&0.11857&0.12821&0.12821&0.12954&\\
8&0.08439&0.02558&0.01569&0.01503&0.01497&0.01497&3.11318\\
16&0.08483&0.02148&0.00612&0.00333&0.00307&0.00306&2.29090\\
32&0.08388&0.02101&0.00528&0.00138&5.40e-04&4.35e-04&2.81364\\
$EOC_T$&&1.99728&1.99302&1.93063&1.35955&0.31024&
\end{tabular}
\caption{\label{TableLinfL2-good2}
Maximal nodal error $\|u^n-U_h^n\|_{L^2}$, where $U_h$ is the solution of   \eqref{MixedFullDivp2}.
The estimated temporal (spacial) order of convergence $EOC_T$ ($EOC_S$) is computed using the last row (column).
}
\begin{tabular}{r|lllllll}
$N_S$\textbackslash $N$ &4&8&16&32&64&128&$EOC_S$\\
\hline
4&0.03099&0.03173&0.03384&0.03479&0.03479&0.03491&\\
8&0.00420&0.00362&0.00361&0.00361&0.00361&0.00361&3.27398\\
16&0.00243&6.05e-04&4.47e-04&4.42e-04&4.42e-04&4.42e-04&3.02991\\
32&0.00234&3.77e-04&7.98e-05&5.34e-05&5.25e-05&5.25e-05&3.07321\\
$EOC_T$&&2.63560&2.24033&0.58021&0.02296&0.00057&
\end{tabular}

\caption{\label{TablePL2L2-good2}
Error in $L^2\otimes L^2$-norm between $p$ and the solution of   \eqref{MixedFullDivp2}.
The estimated temporal (spacial) order of convergence $EOC_T$ ($EOC_S$) is computed using the last row (column).
}
\begin{tabular}{r|lllllll}
$N_S$\textbackslash $N$ &4&8&16&32&64&128&$EOC_S$\\
\hline
4&0.60072&0.85615&1.48345&1.67915&1.67916&1.75676&\\
8&0.26843&0.22768&0.22215&0.24958&0.24960&0.24962&2.81511\\
16&0.21029&0.06563&0.03246&0.03015&0.02991&0.03032&3.04123\\
32&0.21296&0.05310&0.01362&0.00413&0.00257&0.00245&3.63160\\
$EOC_T$&&2.00379&1.96269&1.72067&0.68326&0.07327&
\end{tabular}
}
\end{table}

\section{Conclusion and outlook}\label{SectionOutlook}
We have studied, in an abstract framework, DG time discretization methods for parabolic problems with non-homogeneous linear constraints. Two common ways of treating such linear constraints, namely explicit or implicit (via Lagrange multpiliers), were discussed. These different treatments lead to different variational formulations of the parabolic problem. For these formulations we introduced a \emph{modification} of the standard DG time discretization method in which an appropriate projection is used in the discretization of the constraint; see \eqref{Divp}-\eqref{Divp1} and \eqref{Divpmixed}-\eqref{Divp1mixed}. 
For these discretizations (optimal) discretization error bounds are derived; see Theorem~\ref{mainthm1}, Theorem~\ref{thmSuper}, Theorem~\ref{Thmsuboptimal}, Theorem~\ref{Thmpoptimal}. In the latter two theorems error bounds for the Lagrange multiplier are presented. As far as we know, even for the case with homogeneous constraints, such bounds are not available in the literature, yet.  Numerical experiments confirm the predicted optimal convergence. Furthermore, experiments show that without the modification the (standard) DG method yields sub-optimal results.

We consider the following topics to be of interest for  future research. The fully discrete scheme has merely been introduced in this paper and no error bounds for the fully discrete (in time and space) schemes have been derived. It seems not straightforward how to obtain satisfactory error bounds for the fully discrete schemes. In particular it is not clear how to derive a superconvergence result and  optimal convergence results for the Lagrange multiplier.  Another topic is related to problems where the occurring operators, both in the parabolic equation ($A$ and $(\cdot,\cdot)_\cH$) and in the constraints ($B_1$ and $B_2$), depend on time. It is not clear which convergence results hold (and can be proved) for problems with time-dependent operators, especially ones with low regularity. This includes, for example, operators that arise from multi-phase flow problems. 
\bibliographystyle{siam}
\bibliography{../DG}

 \end{document}